\documentclass[conference]{asl}
\usepackage{rotate}

\def\confname{}
\def\copyrightyear{2020}

\newtheorem{theorem}{theorem}[section]
\newtheorem{prop}[theorem]{Proposition}
\newtheorem{notation}[theorem]{Notation}
\newtheorem{lemma}[theorem]{Lemma}
\newtheorem{cor}[theorem]{Corollary}

\usepackage{stmaryrd}
\usepackage{comment}

\usepackage{tikz} 
\usepackage{pgf} 
\usetikzlibrary{patterns,automata,arrows,shapes,snakes,topaths,trees,backgrounds,positioning,through,calc}

\newtheorem{fact}[theorem]{Fact}

\theoremstyle{definition}
\newtheorem{definition}[theorem]{Definition}
\newtheorem{remark}[theorem]{Remark}
\newtheorem{example}[theorem]{Example}

\usepackage{scalerel,stackengine}
\stackMath
\newcommand\reallywidehat[1]{%
\savestack{\tmpbox}{\stretchto{%
  \scaleto{%
    \scalerel*[\widthof{\ensuremath{#1}}]{\kern-.6pt\bigwedge\kern-.6pt}%
    {\rule[-\textheight/2]{1ex}{\textheight}}
  }{\textheight}%
}{0.5ex}}%
\stackon[1pt]{#1}{\tmpbox}%
}

\usepackage{mathrsfs}

\newcommand{\one}{1}
\newcommand{\zero}{0}

\title{Choice-Free Stone Duality}

\thanks{Postprint of the article in \textit{The Journal of Symbolic Logic}, Vol.~85, No.~1, March 2020, 109-148. This version updates the label on the top arrow in Figure 2, the phrasing of the proof of Proposition 9.1, a reference in Example 8.22 from `Lemma 9.2' to `Lemma~8.21', and a reference in \S~10.1 from `Definition 3.10' to `Definition 3.8'.}

\author{Nick Bezhanishvili}

\address{Institute for Logic, Language and Computation\\
University of Amsterdam\\
1090 GE Amsterdam, The Netherlands}

\email{n.bezhanishvili@uva.nl}

\author{Wesley H. Holliday}

\address{Department of Philosophy and\\
Group in Logic and the Methodology of Science\\
University of California, Berkeley\\
Berkeley, CA 94720, USA}

\email{wesholliday@berkeley.edu}



\begin{document}

\begin{abstract} The standard topological representation of a Boolean algebra via the clopen sets of a Stone space requires a nonconstructive choice principle, equivalent to the Boolean Prime Ideal Theorem. In this paper, we describe a choice-free topological representation of Boolean algebras. This representation uses a subclass of the spectral spaces that Stone used in his representation of distributive lattices via compact open sets. It also takes advantage of Tarski's observation that the regular open sets of any topological space form a Boolean algebra.  We prove without choice principles that any Boolean algebra arises from a special spectral space $X$ via the compact regular open sets of $X$; these sets may also be described as those that are both compact open in $X$ and regular open in the upset topology of the specialization order of $X$, allowing one to apply to an arbitrary Boolean algebra simple reasoning about regular opens of a separative poset. Our representation is therefore a mix of Stone and Tarski, with the two connected by Vietoris: the relevant spectral spaces also arise as the hyperspace of nonempty closed sets of a Stone space endowed with the upper Vietoris topology. This connection makes clear the relation between our point-set topological approach to choice-free Stone duality, which may be called the hyperspace approach, and a point-free approach to choice-free Stone duality using Stone locales. Unlike Stone's representation of Boolean algebras via Stone spaces, our choice-free topological representation of Boolean algebras does not show that every Boolean algebra can be represented as a field of sets; but like Stone's representation, it provides the benefit of a topological perspective on Boolean algebras, only now without choice.  In addition to representation, we establish a choice-free dual equivalence between the category of Boolean algebras with Boolean homomorphisms and a subcategory of the category of spectral spaces with spectral maps. We show how this duality can be used to prove some basic facts about Boolean~algebras.
\end{abstract}

\maketitle

\tableofcontents

\section{Introduction}\label{Intro}

Stone \cite{Stone1936} proved that any Boolean algebra (BA) $\mathbb{A}$ is isomorphic to the field of clopen sets of a Stone space (zero-dimensional compact Hausdorff space), namely the Stone dual of $\mathbb{A}$. As the Stone dual of $\mathbb{A}$ is the set of ultrafilters of $\mathbb{A}$ with the topology generated by $\{\widehat{a}\mid a\in \mathbb{A}\}$, where $\widehat{a}$ is the set of ultrafilters containing $a$, Stone's representation requires a nonconstructive choice principle---equivalent to the Boolean Prime Ideal Theorem---asserting the existence of sufficiently many ultrafilters. 

In this paper, we describe a choice-free topological representation of BAs. This representation uses a subclass of the spectral spaces that Stone \cite{Stone1938} used in his representation of distributive lattices via compact open sets.  It also takes advantage of Tarski's \cite{Tarski1937,Tarski1938} observation that the regular open sets of any topological space form a Boolean algebra.  We prove without choice principles that any Boolean algebra arises from a special spectral space $X$ via the compact regular open sets of $X$; these sets may also be described as those that are both compact open in $X$ and regular open in the upset topology of the specialization order of $X$, allowing one to apply to an arbitrary BA simple reasoning about regular opens of a separative poset.\footnote{The consideration of two topologies is clearly related to Priestley's \cite{Priestley1970} alternative representation for distributive lattices using certain ordered Stone spaces: any distributive lattice arises from a Priestley space via the sets that are both clopen in the Stone topology of the space and open in the upset topology arising from the additional order. We consider a Priestley-like version of our representation of BAs in Section \ref{ChoiceSection}.}  Our representation is therefore a mix of Stone and Tarski, with the two connected by Vietoris~\cite{Vietoris1922}: the relevant spectral spaces also arise as the hyperspace of nonempty closed sets of a Stone space endowed with the upper Vietoris topology. We characterize these spectral spaces, which we call UV-spaces, with several axioms including a special separation axiom, reminiscent of the Priestley separation axiom \cite{Priestley1970}. The connection with the Vietoris hyperspace construction makes clear the relation between our point-set topological approach to choice-free Stone duality, which may be called the hyperspace approach, and a point-free approach to choice-free Stone duality using Stone locales \cite{Johnstone1982,Vickers1989}. 

Unlike Stone's representation of BAs via Stone spaces, the choice-free topological representation of BAs via UV-spaces does not show that every BA can be represented as a field of sets, with complement as set-theoretic complement and join as union. Such a representation implies the Boolean Prime Ideal Theorem.\footnote{If a BA is isomorphic to a field $\mathcal{F}$ of sets over a set $X$, then picking any point $x\in X$ gives us an ultrafilter $\{S\in\mathcal{F}\mid x\in S\}$. The statement that every BA contains an ultrafilter then implies that for any disjoint filter-ideal pair in a BA, the filter can be extended to an ultrafilter disjoint from the ideal. The equivalent dual statement for ideals is the Boolean Prime Ideal Theorem.} However, like Stone's representation, ours provides the benefit of a topological perspective on BAs, only now without choice. 

In addition to representation, we establish a choice-free dual equivalence between the category of BAs with Boolean homomorphisms and the category of UV-spaces with special spectral maps. We show how this duality can be applied by using it to prove some basic theorems about BAs.

The axiom of choice and its variants have traditionally been of general interest to logicians. Interest in choice also arises specifically in connection with topology and Stone duality as in \cite{Johnstone1981,Johnstone1982,Johnstone1983}. In this paper, we assume the motivations summarized in \cite{Herrlich2006} for investigating mathematics without the axiom of choice---in particular, mathematics based on ZF set theory instead of only ZFC. Only starting in our applications section (Section 8) will we go beyond ZF by using the axiom of dependent choice (DC), which is widely considered to be constructively acceptable (see \cite[\S~14.76]{Schechter1996}). There we work in the style of what is called \textit{quasiconstructive mathematics} in \cite{Schechter1996}, defined as ``mathematics that permits conventional rules of reasoning plus ZF + DC, but no stronger forms of Choice'' (p.~404).

The paper is organized as follows. Sections \ref{PossSection} and \ref{RepresentationSection} present requisite background and the representation to be used in the following sections, which is redescribed in Section \ref{ROsection}. Section \ref{VietorisSection} characterizes the resulting duals of BAs as UV-spaces; Section \ref{DualitySection} establishes the dual equivalence result; and Section \ref{L&H} contrasts our hyperspace approach with a localic approach. Section \ref{DictionarySection} contains a ``duality dictionary'' for translating between BA notions and UV notions, and Section \ref{ApplicationSection} contains sample applications of the duality. Although our focus is on choice-free duality, Section \ref{ChoiceSection} considers three perspectives on UV-spaces assuming choice. Section \ref{Conclusion} concludes with a brief recap and look ahead.

\section{Background}\label{PossSection}

The choice-free topological representation of BAs that we will describe results from ``topologizing'' the choice-free representation of BAs in \cite{Holliday2015,Holliday2018}.\footnote{The focus of \cite{Holliday2015,Holliday2018} is on modal algebras, but here we present only the Boolean side of the story.} A \textit{possibility frame} from \cite{Holliday2015} is a triple $(S,\leqslant, P)$ where $(S,\leqslant)$ is a poset and $P$ is a collection of \textit{regular open} sets in the upset topology $\mathsf{Up}(S,\leqslant)$ of the poset, such that $P$ contains $S$ and is closed under intersection and the operation $\neg$ defined by 
\begin{equation}
\neg U=\{x\in S\mid \forall x'\geqslant x\;\; x'\not\in U\}.\label{NegationEQ}
\end{equation}
Recall that an open set $U$ in a space is regular open iff $U=\mathsf{int}(\mathsf{cl}(U))$. Since the closure and interior operations in $\mathsf{Up}(S,\leqslant)$ are calculated by
\begin{eqnarray}
&&\mathsf{cl}_\leqslant (U)=\{x\in S\mid \exists y\geqslant x : y\in U\}, \label{IntDef}\\
&&\mathsf{int}_\leqslant (U)=S\setminus \mathsf{cl}_\leqslant (S\setminus U)= \{x\in S\mid \forall y\geqslant x \;\; y\in U\},\label{ClDef}
\end{eqnarray}
an open set $U$ in $\mathsf{Up}(S,\leqslant)$ is regular open iff
\begin{equation}U=\mathsf{int}_\leqslant(\mathsf{cl}_\leqslant(U))=\{x\in S \mid \forall x'\geqslant x\,\exists x''\geqslant x': x''\in U\}. \label{ROeq}\end{equation}
Also note that $\neg U=\mathsf{int}_\leqslant (X\setminus U)$, so $U$ is regular open iff $U=\neg\neg U$.

As Tarski \cite{Tarski1937,Tarski1938,Tarski1956} observed, the regular open sets of any topological space form a (complete) Boolean algebra with binary meet as intersection and complement as interior of set-theoretic complement, so any subalgebra thereof is also a Boolean algebra. Thus, for any possibility frame $(S,\leqslant,P)$, the set $P$ gives us a Boolean algebra.\footnote{\label{LocaleNote}From the perspective of locale theory (see Section \ref{L&H}), the collection $\mathsf{Up}(S,\leqslant)$ of upsets forms a locale with meet as intersection and join as union. Equivalently, $\mathsf{Up}(S,\leqslant)$ may be viewed as a complete Heyting algebra. Then $\neg U$ is the \textit{pseudocomplement} of $U$ in $\mathsf{Up}(S,\leqslant)$, i.e., the largest upset whose meet with $U$ is $\varnothing$, and $P$ is a subalgebra of the Boolean algebra of all \textit{regular elements} (i.e., those $U$ such that $U=\neg\neg U$) of  $\mathsf{Up}(S,\leqslant)$.}

Conversely, given any Boolean algebra $\mathbb{A}$, we construct a possibility frame $(\mathrm{PropFilt}(\mathbb{A}),\subseteq,\{\widehat{a}\mid a\in \mathbb{A}\})$ where $\mathrm{PropFilt}(\mathbb{A})$ is the set of proper filters of $\mathbb{A}$, ordered by inclusion, and $\widehat{a}=\{F\in \mathrm{PropFilt}(\mathbb{A})\mid a\in F\}$; then $\{\widehat{a}\mid a\in \mathbb{A}\}$ is a collection of regular open sets from  $\mathsf{Up}(\mathrm{PropFilt}(\mathbb{A}),\subseteq)$ that satisfies the required closure conditions, and under the operations $\cap$ and $\neg$ it becomes a Boolean algebra isomorphic to $\mathbb{A}$.\footnote{It can then be proved choice-free that the complete BA of all regular opens from $\mathsf{Up}(\mathrm{PropFilt}(\mathbb{A}),\subseteq)$ is a \textit{canonical extension} of $\mathbb{A}$ in the sense of \cite{Gehrke2001} (see \cite[\S~5.6]{Holliday2018} and Theorem \ref{CanonicalExt} below).} The possibility frames that arise (isomorphically) in this way, called \textit{filter-descriptive} in \cite{Holliday2015}, are exactly those satisfying the separation property that if $x\not\leqslant y$, then there is a $U\in P$ such that $x\in U$ and $y\not\in U$, and the ``filter realization'' property that if $F$ is a proper filter in $P$, then $F=\{U\in P\mid x\in U\}$ for some $x\in S$. In \cite{Holliday2015,Holliday2018} it is proved without choice principles that the category of filter-descriptive frames with appropriate morphisms (see Section \ref{DualitySection}) is dually equivalent to the category of BAs with Boolean homomorphisms.

In Section \ref{RepresentationSection}, we will show that the duality just sketched can be understood topologically as a choice-free duality between BAs and special \textit{spectral spaces}. In particular, the dual possibility frame $(S,\leqslant, P)$ of a BA gives rise to a spectral space $X$ by using $P$ as a basis for a topology on $S$.  This makes $\leqslant$ the specialization order of $X$. We can then conveniently pick out among all regular opens in the upset topology of $\leqslant$ just those that give us back our original BA via $P$: those that are also \textit{compact open} in $X$. It turns out we may equivalently think of these compact sets as regular open in $X$, though thinking of them as regular open in the upset topology of $\leqslant$ has the advantage of simplifying reasoning.  The story above is our starting point, but we go much further: we develop a full topological duality, including a duality dictionary for many algebraic concepts, along with sample applications via topological proofs of basic facts about BAs.

There are several precedents for the strategy of working with all proper filters of a lattice. In the context of logic, since the early 1980s logicians have studied alternative semantics for classical first-order logic and classical modal logics in which one builds a canonical model using all consistent and deductively closed sets of formulas, rather than only maximally consistent sets of formulas \cite{Roper1980,Humberstone1981,Benthem1981,Benthem1986,Benthem1988,Holliday2015,Benthem2016b}. Although not presented as such, these constructions are essentially applications of the fact indicated above that any BA $\mathbb{A}$ embeds into the BA of regular open upsets in the poset of proper filters of $\mathbb{A}$. If $\mathbb{A}$ is the Lindenbaum-Tarski algebra of a logic, then its poset of proper filters is isomorphic to the poset of consistent and deductively closed sets of formulas. The subsets of this canonical model that are definable by a formula then correspond to the sets $\widehat{a}$ above.

The idea of topologizing the set of proper filters also appears in Goldblatt's \cite{Goldblatt1975} representation of ortholattices, discussed in Section \ref{GoldblattSection}. However, Goldblatt uses a different topology on the set of proper filters with the consequence that his representation is not choice free. 

After completing the following work, we learned that Moshier and Jipsen \cite{Moshier2014} propose a choice-free duality for arbitrary lattices using the space of all filters endowed with the analogous $\widehat{a}$ topology. Though we work with proper filters (since otherwise there would  be only two regular open sets with respect to $\leqslant$, namely $\varnothing$ and the whole space), the more important difference is that we study what happens in the special case of BAs.

Our approach to choice-free Stone duality for BAs is also closely related to a point-free approach. The collection $\mathrm{Filt}(\mathbb{A})$ of all filters of a BA $\mathbb{A}$ ordered by inclusion is an example of what we will call a \textit{Stone locale}: a zero-dimensional compact locale (see Section \ref{L&H} for definitions). The category of Stone locales with localic maps\footnote{For the definition of localic maps, see, e.g., \cite[\S~II.2]{Picado2012}.} is dually equivalent to the category of BAs with Boolean homomorphisms. However, our aim is to provide a choice-free duality using spaces instead of locales. We do so by taking the non-zero elements of the Stone locale $\mathrm{Filt}(\mathbb{A})$ as the points of a new space with an appropriate topology, namely the upper Vietoris topology (see Section \ref{RepresentationSection}). Thus, we call our approach to choice-free Stone duality the \textit{hyperspace approach}, in contrast to the \textit{localic approach} using Stone locales.

The hyperspace approach allows us to retain the intuitiveness of reasoning with a set of points, without paying the price of choice principles. But there is a cost, or at least a currency exchange: whereas standard Stone duality represents each BA as a subalgebra of the powerset of a set, the choice-free dualities in \cite{Holliday2015} and in this paper represent each BA as a subalgebra of the regular open algebra of a separative poset. 

\begin{definition} Let $(S,\leqslant)$ be a poset, and for $x\in S$, let $\mathord{\Uparrow}x=\{x'\in S\mid x\leqslant x'\}$. Then $(S,\leqslant)$ is \textit{separative} iff for any $x,y\in S$, $x\not\leqslant y$ implies that there is a $z\in \mathord{\Uparrow}y$ such that $\mathord{\Uparrow}z\cap\mathord{\Uparrow}x=\varnothing$. Equivalently, $(S,\leqslant)$ is separative iff every principal upset $\mathord{\Uparrow}x$ is regular open in $\mathsf{Up}(S,\leqslant)$.
\end{definition}
\noindent It is easy to see that the separation property mentioned for possibility frames above implies separativity of the underlying partial order.

Thus, with the choice-free duality for BAs that we will pursue, instead of reasoning about sets with intersection and set-theoretic complement, we reason about separative posets (given by the specialization orders of our spaces) with intersection and the operation $\neg$ defined in (\ref{NegationEQ}). A major difference is that for $U\subseteq S$, while $U\cup (S\setminus U)=S$, we often have ${U\cup \neg U\subsetneq S}$.\footnote{From the perspective of Footnote \ref{LocaleNote}, the observation that we often have ${U\cup \neg U\subsetneq S}$ reflects the fact that $\mathsf{Up}(S,\leqslant)$ is a Heyting algebra that is typically not Boolean.} This makes reasoning with $\neg$ more subtle, but one can quickly get used to reasoning patterns with $\neg$ of the kind shown in the following lemmas.

\begin{lemma} Let $(S,\leqslant)$ be a poset and $U$ regular open in $\mathsf{Up}(S,\leqslant)$. If $x\not\in U$, then there is an $x'\geqslant x$ such that $x'\in \neg U$.
\end{lemma}
\begin{proof} If $x\not\in U$, then since $U$ is regular open, it follows by (\ref{ROeq}) that there is an $x'\geqslant x$ such that for all $x''\geqslant x'$, $x''\not\in U$, which means $x'\in \neg U$.
\end{proof}

\begin{lemma}\label{EitherInfinite} Let $(S,\leqslant)$ be an infinite separative poset and $U$ regular open in $\mathsf{Up}(S,\leqslant)$. Then either $U$ or $\neg U$ is infinite.
\end{lemma}
\begin{proof} Let $x\sim y$ iff $\mathord{\Uparrow}x\cap U=\mathord{\Uparrow}y\cap U$. If $U$ is finite, then $\sim$ partitions the infinite set $S$ into finitely many cells, one of which must be infinite. Call it $I$, and define $f\colon I\to \wp(\neg U)$ by $f(x)=\mathord{\Uparrow}x\cap \neg U$. We claim that $f$ is injective. For if $x,y\in I$ and $x\not\leqslant y$, then by separativity, there is a $z\in \mathord{\Uparrow}y$ such that $\mathord{\Uparrow}z\cap\mathord{\Uparrow}x=\varnothing$. It follows, since $\mathord{\Uparrow}x\cap U=\mathord{\Uparrow}y\cap U$, that $\mathord{\Uparrow}z\cap U=\varnothing$, so $z\in \neg U$. Thus, $z\in f(y)$ but $z\not\in f(x)$, so $f$ is injective. Then since $I$ is infinite, it follows that $\wp(\neg U)$ is infinite and hence $\neg U$ is infinite.\end{proof}

\section{Representation of BAs using spectral spaces}\label{RepresentationSection} Before reviewing spectral spaces, let us fix some notational conventions. 

We will conflate a BA $\mathbb{A}$ and its underlying set, and we will conflate a topological space $X$ and its underlying set, so that we will write, e.g., `$a\in\mathbb{A}$', `$x\in X$', etc. The top and bottom elements of a bounded lattice such as a BA are denoted `$\one$' and `$\zero$', respectively, possibly with subscripts to indicate the relevant algebra. We will often consider filters in a BA, as well as principal upsets in the specialization order of a space. To avoid any confusion about which side a principal filter/upset is on---the algebra side or the space side---we make the following notational distinction.

\begin{notation} \textnormal{Let $\mathbb{A}$ be a BA whose underlying order is $\leq$ and $X$ a space whose specialization order is $\leqslant$. For $a\in\mathbb{A}$ and $x\in X$:
\begin{enumerate}
\item $\mathord{\uparrow} a = \{b\in\mathbb{A}\mid a\leq b \}$ and $\mathord{\downarrow} a = \{b\in\mathbb{A}\mid b\leq a \}$;
\item $\mathord{\Uparrow}x = \{y\in X\mid x\leqslant y\}$ and $\mathord{\Downarrow}x = \{y\in X\mid y\leqslant x\}$.
\end{enumerate}}
\end{notation}
It will also help to distinguish between the built-in complement operation of a BA $\mathbb{A}$ and the operation $\neg$ defined in (\ref{NegationEQ}) of Section~\ref{PossSection}.
\begin{notation} Given a BA $\mathbb{A}$ and a space $X$ whose specialization order is $\leqslant$:
\begin{enumerate}
\item let $-$ be the complement operation in $\mathbb{A}$;
\item let $\neg$ be the operation defined for $U\subseteq X$ by $\neg U=\mathsf{int}_\leqslant (X\setminus U)$.
\end{enumerate}
\end{notation}

It is important to remember that we are distinguishing two interior (resp.~closure) operations associated with a given space $X$.
\begin{notation} \textnormal{For a space $X$ whose specialization order is $\leqslant$:
\begin{enumerate}
\item $\mathsf{int}$ and $\mathsf{cl}$ are the interior and closure operations for $X$;
\item $\mathsf{int}_\leqslant$ and $\mathsf{cl}_\leqslant$ are the interior and closure operations for the upset topology with respect to $\leqslant$, as in (\ref{IntDef})--(\ref{ClDef}) of Section~\ref{PossSection}.
\end{enumerate}}
\end{notation}

As is well known, the operations $\mathsf{int}_\leqslant$ and $\mathsf{cl}_\leqslant$ coincide with $\mathsf{int}$ and $\mathsf{cl}$, respectively, if and only if $X$ is an Alexandroff space.

The following notation will be used throughout.

\begin{notation}\label{COROnotation} Let $X$ be a space. We define the following collections of subsets of $X$:
\begin{enumerate}
\item $\mathsf{O}(X)$ is the collection of sets that are open in $X$;
\item $\mathsf{C}(X)$ is the collection of sets that are compact in $X$;
\item $\mathsf{CO}(X)=\mathsf{C}(X)\cap\mathsf{O}(X)$;
\item $\mathsf{RO}(X)$ is the collection of sets that are regular open in $X$;
\item $\mathsf{CRO}(X)=\mathsf{C}(X)\cap \mathsf{RO}(X)$;
\item $\mathcal{RO}(X)$ is the collection of sets that are regular open in ${\mathsf{Up}(X,\leqslant)}$, where $\leqslant$ is the specialization order of $X$;
\item $\mathsf{O}\mathcal{RO}(X)=\mathsf{O}(X)\cap \mathcal{RO}(X)$;
\item $\mathsf{CO}\mathcal{RO}(X)=\mathsf{CO}(X)\cap \mathcal{RO}(X)$;
\item $\mathsf{Clop}(X)$ is the collection of sets that are clopen in $X$.
\end{enumerate}
\end{notation}

Let us now recall the notion of a spectral space and two theorems illustrating its importance.

\begin{definition}\label{SpectralDef} A topological space $X$ is a \textit{spectral space} if $X$ is compact, $T_0$, coherent ($\mathsf{CO}(X)$ is closed under intersection and forms a base for the topology of $X$), and sober (every completely prime filter in $\mathsf{O}(X)$ is $\mathsf{O}(x)=\{U\in \mathsf{O}(X)\mid x\in U\}$ for some $x\in X$).
\end{definition}

\begin{theorem}[Stone \cite{Stone1938}] $L$ is a distributive lattice iff $L$ is isomorphic to the lattice of compact open sets of a spectral space.
\end{theorem}

\begin{theorem}[Hochster \cite{Hochster1969}] $X$ is a spectral space iff $X$ is homeomorphic to the spectrum of a commutative ring.
\end{theorem}

We will show that every BA $\mathbb{A}$ can be represented as $\mathsf{CO}\mathcal{RO}(X)$ (or equivalently $\mathsf{CRO}(X)$, as shown in Section \ref{ROsection}) for some spectral space $X$. Using the nonconstructive Boolean Prime Ideal Theorem, one could prove this by taking $X$ to be the Stone space of $\mathbb{A}$: since the specialization order $\leqslant$ in a Stone space is the discrete order, all subsets are regular open in $\mathsf{Up}(X,\leqslant)$, and it can be proved that the compact open sets of $X$ are exactly the clopen sets used in the standard Stone representation. However, it is also possible to provide a choice-free representation, as shown below.

We first recall the \textit{upper Vietoris topology} \cite{Vietoris1922} on the hyperspace of nonempty closed sets of a Stone space. Where $\mathsf{F}(X)$ is the collection of nonempty closed subsets of $X$ and $U\in\mathsf{Clop}(X)$, let \[\Box U =\{F\in \mathsf{F}(X)\mid F\subseteq U\}.\]
Observe that $\Box U\cap \Box V = \Box (U\cap V)$, so $\{\Box U\mid U\in\mathsf{Clop}(X)\}$ is closed under binary intersection.

\begin{definition}\label{UVStoneDef} Given a Stone space $X$, define $\mathscr{UV}(X)$ to be the space of nonempty closed sets of $X$ with the topology generated by the family $\{\Box U\mid U\in\mathsf{Clop}(X)\}$. 
\end{definition}

The same idea can be applied to the space of proper filters of a BA. For $a\in \mathbb{A}$, let \[\widehat{a} = \{F\in \mathrm{PropFilt}(\mathbb{A})\mid a\in F\}.\] Observe that $\widehat{a}\cap\widehat{b}=\widehat{a\wedge b}$, so $\{\widehat{a}\mid a\in \mathbb{A}\}$ is closed under binary intersection.

\begin{definition}\label{UVofBA} Given a BA $\mathbb{A}$, define $UV(\mathbb{A})$ to be the space of proper filters of $\mathbb{A}$ with the topology generated by $\{\widehat{a}\mid a\in \mathbb{A}\}$.
\end{definition}

\begin{prop}\label{UVStone}  For any Stone space $X$, $\mathscr{UV}(X)$ is homeomorphic to $UV(\mathsf{Clop}(X))$, regarding $\mathsf{Clop}(X)$ as the BA of clopen subsets of $X$.
\end{prop}

\begin{proof} Let $f: C\mapsto \{U\in\mathsf{Clop}(X)\mid C\subseteq U\}$. Since $X$ is nonempty, $f(C)$ is clearly a proper filter in $\mathsf{Clop}(X)$, so $f(C)\in UV(\mathsf{Clop}(X))$. For injectivity, if $C\neq C'$, then without loss of generality suppose $x\in C\setminus C'$. Since $X$ is compact Hausdorff, it follows that there is a $U\in\mathsf{Clop}(X)$ such that $C'\subseteq U$ but $x\not\in U$, so $C\not\subseteq U$. Hence $U\in f(C')$ but $U\not\in f(C)$. For surjectivity, if $F$ is a proper filter in $\mathsf{Clop}(X)$, then $F$ has the finite intersection property, so by the compactness of $X$, we have that $\bigcap F$ is nonempty, and since $\bigcap F$  is the intersection of closed sets, it is closed. We claim that $f(\bigcap F)=F$. That $f(\bigcap F)=\{U\in\mathsf{Clop}(X)\mid \bigcap F\subseteq U\}\supseteq F$ is immediate. To see that $f(\bigcap F)\subseteq F$, if $U\in\mathsf{Clop}(X)$ and $\bigcap F\subseteq U$, so $X\setminus U\subseteq \bigcup \{X\setminus V\mid V\in F\}$, then by compactness there is a finite $F_0\subseteq F$ such that $X\setminus U\subseteq \{X\setminus V\mid V\in F_0\}$ and hence $\bigcap F_0\subseteq U$. Then since $F_0$ is finite, it follows that $U\in F$. For continuity of $f$, if $\widehat{U}$ is a basic open in $UV(\mathsf{Clop}(X))$, so $U\in\mathsf{Clop}(X)$, then we have:
\begin{eqnarray*}
f^{-1}[\widehat{U}]&=&\{C\in UV(X)\mid f(C)\in \widehat{U}\}\\
&=&\{C\in UV(X)\mid U\in f(C)\}\\
&=&\{C\in UV(X)\mid C\subseteq U\}\\
&=& \Box U.
\end{eqnarray*}
For openness of $f$, if $\Box U$ is a basic open in $\mathscr{UV}(X)$, so $U\in\mathsf{Clop}(X)$, then we have:
\begin{eqnarray*}
f[\Box U]&=&\{f(C)\mid C\in \Box U\} \\
&=& \{f(C)\mid C\subseteq U\} \\
&=& \{f(C)\mid U\in f(C)\} \\
&=& \widehat{U}.
\end{eqnarray*}
For the last equality, the left-to-right inclusion follows from the fact that $f(C)$ is a proper filter, since $C\neq \varnothing$, and the right-to-left inclusion follows from the surjectivity of $f$.
\end{proof}

\begin{remark} Assuming the Boolean Prime Ideal Theorem, one can also prove that for any BA $\mathbb{A}$, $UV(\mathbb{A})$ is homeomorphic to $\mathscr{UV}(\mathrm{Stone}(\mathbb{A}))$, where $\mathrm{Stone}(\mathbb{A})$ is the Stone dual of $\mathbb{A}$ (see Section~\ref{UVStoneSection}).
\end{remark}

\begin{prop}\label{IsSpectral} For any BA $\mathbb{A}$:
\begin{enumerate}
\item\label{Spectral} $UV(\mathbb{A})$ is a spectral space;
\item\label{Specialization} the specialization order in $UV(\mathbb{A})$ is the inclusion order.
\end{enumerate}
\end{prop}

\begin{proof} We first show that each $\widehat{a}$ is compact open in $UV(\mathbb{A})$. Since the sets $\widehat{b}$ form a basis, it suffices to show that if $\widehat{a}\subseteq \underset{i\in I}{\bigcup}\widehat{b_i}$, then there is a finite subcover. If $\widehat{a}\subseteq \underset{i\in I}{\bigcup}\widehat{b_i}$, then every proper filter that contains $a$ also contains one of the $b_i$. In particular, the principal filter $\mathord{\uparrow}a$ contains one of the $b_i$, which implies $a\leq b_i$ and hence $\widehat{a}\subseteq \widehat{b_i}$, so $\widehat{b_i}$ alone is the finite subcover. It follows that $UV(\mathbb{A})$ is compact, since $X=\widehat{\one}$. It also follows by the definition of $UV(\mathbb{A})$ that the compact open sets form a basis. 

To see that the compact opens are closed under binary intersection, suppose $U$ and $V$ are compact open, so $U=\underset{i\in I}{\bigcup}\widehat{a_i}$ and $V=\underset{j\in J}{\bigcup}\widehat{b_j}$ for finite $I$ and $J$. Then 
\[U\cap V= \underset{i\in I,\,j\in J}\bigcup (\widehat{a_i}\cap\widehat{b_j})=\underset{i\in I,\,j\in J}\bigcup \widehat{a_i\wedge b_j},\]
which is a finite union of compact opens. Hence $U\cap V$ is compact open.

For $T_0$, if $F\neq F'$, without loss of generality suppose $a\in F\setminus F'$. Then $F\in \widehat{a}$ but $F'\not \in \widehat{a}$, and $\widehat{a}$ is open, so we are done.

For sobriety, we show that every completely prime filter $\mathcal{F}$ in $\mathsf{O}(UV(\mathbb{A}))$ is of the form $\mathsf{O}(F)=\{U\in \mathsf{O}(UV(\mathbb{A}))\mid F\in U\}$ for some $F\in UV(\mathbb{A})$. Let $F$ be the filter generated by $\{a\in\mathbb{A}\mid \widehat{a}\in\mathcal{F}\}$. Then since $\mathcal{F}$ is a proper filter in $\mathsf{O}(UV(\mathbb{A}))$, it follows that $F$ is a proper filter in $\mathbb{A}$. To see that $\mathcal{F}=\mathsf{O}(F)$, the right-to-left direction is immediate from the definition of $F$. For the left-to-right direction, suppose $U=\underset{i\in I}{\bigcup}\widehat{a_i}\in\mathcal{F}$. Then since $\mathcal{F}$ is completely prime, there is an $a_i$ such that $\widehat{a_i}\in\mathcal{F}$, which implies $a_i\in F$, so $F\in\widehat{a_i}$. Thus, $\widehat{a_i}\in \mathsf{O}(F)$ and hence $U\in\mathsf{O}(F)$.

For part \ref{Specialization}, we already saw above for $T_0$ that if $F\not\subseteq F'$, then $F\not\leqslant F'$. Conversely, if $F\subseteq F'$, then for any basic open $\widehat{a}$, if $F\in\widehat{a}$ and hence $a\in F$, then $a\in F'$ and hence $F'\in\widehat{a}$, so $F\leqslant F'$.
\end{proof}

We now provide the promised choice-free representation.

\begin{theorem}\label{MainRep}$\,$
\begin{enumerate}
\item\label{MainRep1} For each BA $\mathbb{A}$, the map $\widehat{\cdot}:\mathbb{A}\to \mathsf{CO}\mathcal{RO}(UV(\mathbb{A}))$ is an isomorphism from $\mathbb{A}$ to $\mathsf{CO}\mathcal{RO}(UV(\mathbb{A}))$ ordered by inclusion.
\item\label{MainRep2}  $\mathsf{CO}\mathcal{RO}(UV(\mathbb{A}))$ is a BA with operations given by:
\begin{equation}  U\wedge V=U\cap V\quad \mathord{-}U=\mathsf{int}_\leqslant(UV(\mathbb{A})\setminus U)\quad U\vee V=\mathsf{int}_\leqslant(\mathsf{cl}_\leqslant(U\cup V)).\label{MainRepEQ}\end{equation}
\end{enumerate}
\end{theorem}

\begin{proof} For part \ref{MainRep1}, we will show that
\begin{equation}\mathsf{CO}\mathcal{RO}(UV(\mathbb{A}))=\{\widehat{a}\mid a\in\mathbb{A}\},\label{MainEq}\end{equation}
for then the map $a\mapsto \widehat{a}$ is the isomorphism from $\mathbb{A}$ to $\mathsf{CO}\mathcal{RO}(UV(\mathbb{A}))$, since clearly $a\leq b$ iff $\widehat{a}\subseteq \widehat{b}$. 

For the right-to-left inclusion of (\ref{MainEq}), we showed in the proof of Proposition \ref{IsSpectral}.\ref{Spectral} that each $\widehat{a}$ is compact open in $UV(\mathbb{A})$. Now we show that $\widehat{a}$ is regular open in $\mathsf{Up}(UV(\mathbb{A}),\leqslant)$, using the fact from Proposition \ref{IsSpectral}.\ref{Specialization} that the specialization order $\leqslant$ is the inclusion order $\subseteq$. First, $\widehat{a}$ is an $\leqslant$-upset, for if $F\in \widehat{a}$ and $F\leqslant F'$, so $a\in F$ and $F\subseteq F'$, then $a\in F'$ and hence $F'\in \widehat{a}$. Then to see that $\widehat{a}$ is regular open, by (\ref{ROeq}) it suffices to show that if $F\not\in \widehat{a}$, then there is a proper filter $F'\supseteq F$ such that for all proper filters $F''\supseteq F'$, we have $F''\not\in \widehat{a}$. Indeed, if $F\not\in \widehat{a}$, so $a\not\in F$, then the filter $F'$ generated by $F\cup \{-a\}$ is a proper filter with $F'\supseteq F$, and for all proper filters $F''\supseteq F'$, we have $a\not\in F''$ and hence $F''\not\in \widehat{a}$.

For the left-to-right inclusion of (\ref{MainEq}), suppose $S$ is compact open, so $S=\widehat{a_1}\cup\dots\cup \widehat{a_n}$ for some $a_1,\dots,a_n\in\mathbb{A}$. Now if in addition $\widehat{a_1}\cup\dots \cup \widehat{a_n}$ is regular open in $\mathsf{Up}(UV(\mathbb{A}),\leqslant)$, then we claim 
\begin{equation}\widehat{a_1}\cup\dots\cup  \widehat{a_n}=\reallywidehat{a_1\vee\dots\vee a_n}.\label{FiniteJoinEq}\end{equation} 
First, we show 
\begin{equation}\reallywidehat{a_1\vee\dots\vee a_n}=\mathsf{int}_\leqslant(\mathsf{cl}_\leqslant(\widehat{a_1}\cup\dots\cup \widehat{a_n})).\label{JoinEQ}\end{equation}
For the left-to-right inclusion, if $F\in \reallywidehat{a_1\vee\dots\vee a_n}$, so $a_1\vee\dots \vee a_n\in F$, then for any proper filter $F'\supseteq F$, there is some $a_i$ such that $-a_i\not\in F'$. Thus, the filter $F''$ generated by $F'\cup \{a_i\}$ is proper, and $a_i\in F''$ implies $F''\in \widehat{a_i}$ and hence $F''\in\widehat{a_1}\cup\dots\cup \widehat{a_n}$. Thus, by (\ref{ROeq}), $F\in \mathsf{int}_\leqslant(\mathsf{cl}_\leqslant(\widehat{a_1}\cup\dots\cup \widehat{a_n}))$.  Conversely, if $F\not\in \reallywidehat{a_1\vee\dots\vee a_n}$, so $a_1\vee\dots\vee a_n\not\in F$, then the filter $F'$ generated by $F\cup\{-a_1\wedge\dots\wedge -a_n\}$ is a proper filter, and for every proper filter $F''\supseteq F'$, each $a_i$ is not in $F''$, so $F''\not\in \widehat{a_1}\cup\dots\cup \widehat{a_n}$. Thus, by (\ref{ROeq}), $F\not\in \mathsf{int}_\leqslant(\mathsf{cl}_\leqslant(\widehat{a_1}\cup\dots\cup \widehat{a_n}))$. Finally, if $\widehat{a_1}\cup\dots \cup \widehat{a_n}$ is regular open in $\mathsf{Up}(UV(\mathbb{A}),\leqslant)$, then $\widehat{a_1}\cup\dots \cup\widehat{a_n}=\mathsf{int}_\leqslant(\mathsf{cl}_\leqslant(\widehat{a_1}\cup\dots \cup\widehat{a_n}))$, which with (\ref{JoinEQ}) implies (\ref{FiniteJoinEq}). Thus, $S\in \{\widehat{a}\mid a\in\mathbb{A}\}$.

For part \ref{MainRep2}, since $a\mapsto\widehat{a}$ is an isomorphism, we have:
\begin{equation}\widehat{a}\wedge \widehat{b}=\widehat{a\wedge b}\qquad \mathord{-}\widehat{a}=\widehat{- a}\qquad \widehat{a}\vee\widehat{b}=\widehat{a\vee b}.\label{From1Eq}\end{equation}
We have already observed the first and third of the following equalities:
\begin{equation}\widehat{a\wedge b}=\widehat{a}\cap\widehat{b}\qquad \widehat{- a}=\mathsf{int}_\leqslant(UV(\mathbb{A})\setminus \widehat{a})\qquad \widehat{a\vee b}=\mathsf{int}_\leqslant(\mathsf{cl}_\leqslant(\widehat{a}\cup\widehat{b})).\label{AlreadyObserved}\end{equation}
For the second equality, if $F\in   \widehat{-a}$, so $-a\in F$, then for every proper filter $F'\supseteq F$, we have $-a\in F'$, so $a\not\in F'$ and hence $F'\not\in\widehat{a}$. Thus, $F\in \mathsf{int}_\leqslant(UV(\mathbb{A})\setminus \widehat{a})$. If $-a\not\in F$, then the filter $F'$ generated by $F\cup \{a\}$ is a proper filter such that $F\subseteq F'\in \widehat{a}$, so $F\not\in \mathsf{int}_\leqslant(UV(\mathbb{A})\setminus \widehat{a})$.

Combining (\ref{From1Eq}) and (\ref{AlreadyObserved}), we have:
\begin{equation}\widehat{a}\wedge \widehat{b}=\widehat{a}\cap\widehat{b}\qquad \mathord{-}\widehat{a}=\mathsf{int}_\leqslant (UV(\mathbb{A})\setminus \widehat{a})\qquad \widehat{a}\vee\widehat{b}=\mathsf{int}_\leqslant (\mathsf{cl}_\leqslant (\widehat{a}\cup\widehat{b})),\end{equation}
which with (\ref{MainEq}) shows that the BA operations of $\mathsf{CO}\mathcal{RO}(UV(\mathbb{A}))$ satisfy the equations in (\ref{MainRepEQ}).
\end{proof}

\begin{cor} For each Stone space $X$,  $\mathsf{Clop}(X)$ is isomorphic to $\mathsf{CO}\mathcal{RO}(\mathscr{UV}(X))$ via the map $U\mapsto \Box U$.
\end{cor}

\begin{proof} By Theorem \ref{MainRep}, we have an isomorphism between $\mathsf{Clop}(X)$ and $\mathsf{CO}\mathcal{RO}(UV(\mathsf{Clop}(X)))$ via the map that sends $U\in\mathsf{Clop}(X)$ to $\widehat{U}\in \mathsf{CO}\mathcal{RO}(UV(\mathsf{Clop}(X)))$. By the proof of Proposition \ref{UVStone}, $\mathscr{UV}(X)$ is homeomorphic to $UV(\mathsf{Clop}(X))$ via the map $f$, which satisfies $f^{-1}[\widehat{U}]=\Box U$. Thus, $\mathsf{Clop}(X)$ is isomorphic to $\mathsf{CO}\mathcal{RO}(\mathscr{UV}(X))$ via the map $U\mapsto \Box U$.
\end{proof}

\section{Regular opens in the Alexandroff and spectral topologies}\label{ROsection} In response to the representation in the previous section, Tom\'{a}\v{s} Jakl (p.~c.) observed that in the special case of compact open sets, being regular open in the Alexandroff space $\mathsf{Up}(UV(\mathbb{A}))$ is equivalent to being regular open in the spectral space $UV(\mathbb{A})$, i.e., $\mathsf{CO}\mathcal{RO}(UV(\mathbb{A}))=\mathsf{CRO}(UV(\mathbb{A}))$. We have $U\in\mathsf{RO}(UV(\mathbb{A}))$  iff $U$ is an open set such that $U=\mathsf{int}(\mathsf{cl}(U))$, where $\mathsf{int}$ and $\mathsf{cl}$ are the interior and closure operations of $UV(\mathbb{A})$. This is equivalent to $U=U^{**}$, where $^*$ is the pseudocomplement operation on $\mathsf{O}(UV(\mathbb{A}))$:
\[U^*=\mathsf{int}(UV(\mathbb{A})\setminus U).\]
It is then easy to see that
\[U^*=\bigcup \{V\in \mathsf{O}(UV(\mathbb{A}))\mid U\cap V=\varnothing\}=\bigcup\{\widehat{c}\mid U\cap\widehat{c}=\varnothing\}.\]
Thus, we can derive $\mathsf{CO}\mathcal{RO}(UV(\mathbb{A}))=\mathsf{CRO}(UV(\mathbb{A}))$ from the following more basic facts.

\begin{prop}\label{StarNeg} Let $\mathbb{A}$ be a BA.
\begin{enumerate}
\item\label{StarNeg1} If $U\in\mathsf{O}(UV(\mathbb{A}))$, then $U^*\subseteq \neg U$;
\item\label{StarNeg2} If $U\in\mathsf{CO}(UV(\mathbb{A}))$, then $\neg U\subseteq U^*$.
\end{enumerate}
\end{prop}
\begin{proof} For part (\ref{StarNeg1}), suppose $F\in U^*$, so there is some $c$ such that $F\in\widehat{c}$, i.e., $c\in F$, and $U\cap\widehat{c}=\varnothing$, i.e., no proper filter containing $c$ belongs to $U$. Thus, no proper filter extending $F$ belongs to $U$, whence $F \in \neg U$.

For part (\ref{StarNeg2}), suppose $U\in\mathsf{CO}(X)$, so $U=\widehat{a_1}\cup\dots\cup\widehat{a_n}$ for some $a_1,\dots,a_n\in\mathbb{A}$. Then assuming $F\in\neg U$, we have $\neg a_1,\dots,\neg a_n\in F$ and hence $c:=\neg a_1\wedge\dots\wedge\neg a_n \in F$. Thus, $F\in\widehat{c}$, and clearly $U\cap\widehat{c}=\varnothing$. Therefore, $F\in U^*$.
\end{proof}

As an immediate corollary of Proposition \ref{StarNeg}, we have the following.

\begin{cor}\label{Jakl} For any BA $\mathbb{A}$, $\mathsf{CO}\mathcal{RO}(UV(\mathbb{A}))=\mathsf{CRO}(UV(\mathbb{A}))$.
\end{cor}
\noindent Thus, by Theorem \ref{MainRep}, $\mathbb{A}$ is isomorphic to $\mathsf{CRO}(UV(\mathbb{A}))$. It is also easy to check that $-\widehat{a}=\mathsf{int}(UV(\mathbb{A})\setminus \widehat{a})$ and $\widehat{a}\vee\widehat{b}=\mathsf{int}(\mathsf{cl}(\widehat{a}\cup \widehat{b}))$.

If we do not restrict to compact open sets, then the operations $\neg$ and $^*$ may behave differently; however, the extent of this difference depends on one's set-theoretic assumptions. It is a theorem of $\mathrm{ZF}+\mathrm{BPI}$ that every infinite BA contains a non-principal ultrafilter (see, e.g., \cite[p.~174]{Givant2009}), in which case $\neg$ and $^*$ can be distinguished with an open set as in Proposition \ref{Nonprinc}.\ref{Nonprinc1} below. On the other hand, it is consistent with ZF that there is an infinite BA in which every filter is principal \cite{Plotkin1976} (for an overview, see \cite[p.~165]{Howard1998}), and in such a BA $\neg$ and $^*$ cannot be distinguished with open sets in light of Proposition \ref{Nonprinc}.\ref{Nonprinc2} (plus Proposition \ref{StarNeg}.\ref{StarNeg1}).

\begin{prop}\label{Nonprinc} Let $\mathbb{A}$ be a BA.
\begin{enumerate}
\item\label{Nonprinc0} $\mathsf{RO}(UV(\mathbb{A}))\subseteq \mathsf{O}\mathcal{RO}(UV(\mathbb{A}))$.
\item\label{Nonprinc1} If $F$ is a non-principal ultrafilter in $\mathbb{A}$ and $U=\bigcup\{\widehat{-a}\mid a\in F\}$, then:
\begin{enumerate}
\item\label{Nonprinc1a} $F\in \neg U\setminus U^*$;
\item\label{Nonprinc1b} $U=\neg\neg U$;
\item\label{Nonprinc1c} $U\subsetneq U^{**}$;
\item\label{Nonprinc1d} $\mathsf{O}\mathcal{RO}(UV(\mathbb{A}))\not\subseteq\mathsf{RO}(UV(\mathbb{A}))$.
\end{enumerate}
\item\label{Nonprinc2} Let $F$ be a principal filter in $\mathbb{A}$ and $U\in\mathsf{O}(UV(\mathbb{A}))$. If $F\in \neg U$, then $F\in U^*$.
\end{enumerate}
\end{prop}
\begin{proof} For part \ref{Nonprinc0}, suppose $U\in \mathsf{RO}(UV(\mathbb{A}))$, so $U=U^{**}$. Since $U^*\in\mathsf{O}(UV(\mathbb{A}))$, we have $U^*=\bigcup\{\widehat{b}\mid b\in B\}$ for some $B\subseteq\mathbb{A}$. Thus, 
\begin{eqnarray*}U^{**}&=&\bigcup\{\widehat{c}\mid \bigcup\{\widehat{b}\mid b\in B \} \cap  \widehat{c}=\varnothing\}\\
&=&\bigcup\{\widehat{c}\mid \forall b\in B\;\, \widehat{b} \cap  \widehat{c}=\varnothing\}\\
&=&\bigcup\{\widehat{c}\mid \forall b\in B\;\, b\wedge c=0\}.
\end{eqnarray*}
Let $I:= \{c\in\mathbb{A}\mid \forall b\in B\;\, b\wedge c=0\}$, and observe that $I$ is an ideal in $\mathbb{A}$. To see that $U^{**}\in \mathcal{RO}(UV(\mathbb{A}))$, suppose $F$ is a proper filter in $\mathbb{A}$ such that $F\not\in U^{**}$. It follows that $F\cap I=\varnothing$.  Let $F'$ be the filter generated by $\{a\wedge -c\mid a\in F, c\in I\}$. We claim that $F'$ is a proper filter. If not, then there are $a_1,\dots,a_n\in F$ and $c_1,\dots,c_n\in I$ such that $a_1\wedge -c_1\wedge\dots \wedge a_n\wedge -c_n=0$, so $a_1\wedge\dots\wedge a_n\leq c_1\vee\dots\vee c_n$. Then since $F$ is a filter containing $a_1,\dots,a_n$, we have $c_1\vee\dots\vee c_n\in F$, and since $I$ is an ideal containing $c_1,\dots,c_n$, we have $c_1\vee\dots\vee c_n\in I$, contradicting $F\cap I=\varnothing$. Hence $F'$ is a proper filter, and clearly every proper filter $F'' \supseteq F'$ is disjoint from $I$, so $F''\not\in U^{**}$. It follows that $F'\in\neg (U^{**})$, which with $F\subseteq F'$ implies $F\not\in \neg\neg (U^{**})$. Thus, $\neg\neg (U^{**})\subseteq U^{**}$, so we have $U^{**}=U\in \mathcal{RO}(UV(\mathbb{A}))$. 

For part (\ref{Nonprinc1a}), clearly $F\in\neg U$. Suppose for contradiction that $F\in U^*$, so there is a $c$ such that $F\in\widehat{c}$ and $U\cap\widehat{c}=\varnothing$. Since $F\in\widehat{c}$, we have $c\in F$. We claim that $F$ is the principal filter generated by $c$, i.e.,  $c\leq a$ for all $a\in F$. For if there is an $a\in F$ such that $c\not\leq a$, then $c\wedge - a\neq 0$, so there is a proper filter $G$ containing $c\wedge - a$. Hence $c,- a\in G$, so $G\in\widehat{c}$ and $G\in\widehat{- a}$. Since $a\in F$, $G\in\widehat{- a}$ implies $G\in U$. Then since $G\in\widehat{c}$, we have $G\in U\cap\widehat{c}$, contradicting $U\cap\widehat{c}=\varnothing$ above. Thus, $F\not\in U^*$.

For part (\ref{Nonprinc1b}), $U\subseteq\neg\neg U$ always holds. To see $\neg\neg U\subseteq U$, suppose $G\not\in U$. It follows by definition of $U$ that for all $a\in F$, $G\not\in \widehat{-a}$ and hence $-a\not\in G$. We claim that $G\subseteq F$. Suppose $b\not\in F$, so $-b\in F$ since $F$ is an ultrafilter. Then by what we derived above, $\mathnormal{--}b\not\in G$, i.e., $b\not \in G$. Thus, $G\subseteq F$. Then since $F\in\neg U$, we have $G\not\in \neg\neg U$.

For part (\ref{Nonprinc1c}), again $U\subseteq U^{**}$ always holds. Recall $U^*=\bigcup\{\widehat{c}\mid U\cap\widehat{c}=\varnothing\}$. Given the definition of $U$, the condition that $U\cap\widehat{c}=\varnothing$ is equivalent to: for all $a\in F$, $\widehat{-a}\cap\widehat{c}=\varnothing$. This is in turn equivalent to: for all $a\in F$, $-a\wedge c=0$, i.e., $c\leq a$. Since $F$ is a non-principal ultrafilter, the only $c$ such that $c\leq a$ for all $a\in F$ is given by $c:=0$. Thus, $U^*=\bigcup \{\widehat{0}\}=\bigcup\{\varnothing\}=\varnothing$. It follows that $U^{**}=UV(\mathbb{A})$. Then since $F\not\in U$, we have $U\subsetneq U^{**}$.

Part (\ref{Nonprinc1d}) is immediate from parts (\ref{Nonprinc1b})--(\ref{Nonprinc1c}).

For part (\ref{Nonprinc2}), since $U$ is open, $U=\bigcup\{\widehat{a_i}\mid i\in I\}$ for some $I$. Assuming $F\in\neg U$, we have $\neg a_i\in F$ for each $i\in I$.  If $F$ is a principal filter generated by some $c$, then $c\leq \neg a_i$ for each $i\in I$, so $U\cap\widehat{c}=\varnothing$. Hence $F\in U^*$.
\end{proof}

\begin{remark} The inclusions \[\mathsf{CO}\mathcal{RO}(UV(\mathbb{A}))= \mathsf{CRO}(UV(\mathbb{A}))\subseteq\mathsf{RO}(UV(\mathbb{A}))\subseteq \mathsf{O}\mathcal{RO}(UV(\mathbb{A}))\] can be understood in terms of the dual correspondence between these types of regular open sets and ideals in the BA $\mathbb{A}$, as we will show in Section \ref{DictionarySection}:
\begin{eqnarray*}
\mathsf{O}\mathcal{RO}(UV(\mathbb{A})) & \quad\mbox{corresponds to}\quad & \mbox{ideals of }\mathbb{A}\\
\mathsf{RO}(UV(\mathbb{A})) & \quad\mbox{corresponds to}\quad & \mbox{normal ideals of }\mathbb{A}\\
\mathsf{CO}\mathcal{RO}(UV(\mathbb{A})) & \quad\mbox{corresponds to}\quad & \mbox{principal ideals of }\mathbb{A}. \\
= \mathsf{CRO}(UV(\mathbb{A}))
\end{eqnarray*}
\end{remark}

Given Theorem \ref{MainRep} and the fact that $\mathsf{CO}\mathcal{RO}(UV(\mathbb{A}))=\mathsf{CRO}(UV(\mathbb{A}))$, we can reason about elements of a BA as compact open sets in $UV(\mathbb{A})$ that are regular open in either the Alexandroff  space $\mathsf{Up}(UV(\mathbb{A}))$ or in the spectral space $UV(\mathbb{A})$. Since the definition of a regular open set in the Alexandroff space is especially simple, given by the first-order condition (\ref{ROeq}) involving the specialization order $\leqslant$, we will continue to use this definition of regular open for the purposes of our calculations.

\section{Characterization of choice-free duals of BAs}\label{VietorisSection}

We now wish to characterize the spectral spaces $X$ that are homeomorphic to $UV(\mathbb{A})$ for some Boolean algebra $\mathbb{A}$. For the following definition, given $x\in X$, let $\mathsf{CO}\mathcal{RO}(x)=\{U\in \mathsf{CO}\mathcal{RO}(X)\mid x\in U\}$.

\begin{definition}\label{VOspace} A \textit{UV-space} is a $T_0$ space $X$ such that:
\begin{enumerate}
\item\label{CloseProp} $\mathsf{CO}\mathcal{RO}(X)$ is closed under $\cap$ and $\mathsf{int}_\leqslant (X\setminus \cdot)$ and is a basis for $X$;
\item\label{PossCompact} every proper filter in $\mathsf{CO}\mathcal{RO}(X)$ is $\mathsf{CO}\mathcal{RO}(x)$ for some $x\in X$.
\end{enumerate}
\end{definition}
\begin{remark} An equivalent definition of a $UV$-space (in light of Section \ref{ROsection} and the proof of Theorem \ref{SecondThm} below) substitutes $\mathsf{CRO}$ for $\mathsf{CO}\mathcal{RO}$ and $\mathsf{int}$ for $\mathsf{int}_\leqslant$ in Definition \ref{VOspace}.
\end{remark}

The conditions in Definition \ref{VOspace} are reminiscent of conditions mentioned earlier: compare part 1 with the statement of coherence in Definition \ref{SpectralDef} and part 2 with the statement of sobriety in Definition \ref{SpectralDef}. Note that the basis condition implies an analogue of the Priestley separation axiom \cite{Priestley1970}: if $x\not\leqslant y$, then there is a $U\in\mathsf{CO}\mathcal{RO}(X)$ such that $x\in U$ and $y\not\in U$.

\begin{prop}\label{COROBA} For any UV-space $X$, $\mathsf{CO}\mathcal{RO}(X)$ ordered by inclusion is a BA with the following operations:
\[U\wedge V=U\cap V\qquad\neg U=\mathsf{int}_\leqslant(X\setminus U)\qquad U\vee V=\mathsf{int}_\leqslant(\mathsf{cl}_\leqslant(U\cup V)).\]
\end{prop}

\begin{proof} As noted in Section \ref{PossSection}, it is a well-known result of Tarski that the collection of all regular open sets of a space forms a BA with the operations $\wedge$, $\neg$, and $\vee$ defined above (see, e.g., \cite[\S~4]{Halmos1963}). By Definition \ref{VOspace}.\ref{CloseProp}, in a UV-space $X$, $\mathsf{CO}\mathcal{RO}(X)$ with the operations $\wedge$ and $\neg$ is a subalgebra of the full regular open algebra and therefore a BA.\end{proof}

We now prove that Definition \ref{VOspace} provides our desired characterization.

\begin{theorem}\label{SecondThm} For any BA $\mathbb{A}$ and space $X$:
\begin{enumerate}
\item\label{SecondThmA} $UV(\mathbb{A})$ is a UV-space;
 \item\label{SecondThmB} $X$ is homeomorphic to $UV(\mathsf{CO}\mathcal{RO}(X))$ iff $X$ is a UV-space.
 \end{enumerate}
\end{theorem}

\begin{proof} For part \ref{SecondThmA}, to see that property \ref{CloseProp} of Definition \ref{VOspace} holds, if $U,V\in\mathsf{CO}\mathcal{RO}(UV(\mathbb{A}))$, then by the proof of Theorem \ref{MainRep} we have that $U=\widehat{a}$ and $V=\widehat{b}$ for some $a,b\in \mathbb{A}$. We also saw in the proof of Theorem \ref{MainRep} that $\widehat{a}\cap \widehat{b}=\widehat{a\wedge b}\in \mathsf{CO}\mathcal{RO}(UV(\mathbb{A}))$ and $\mathsf{int}_\leqslant (UV(\mathbb{A})\setminus \widehat{a})=\widehat{-a}\in \mathsf{CO}\mathcal{RO}(UV(\mathbb{A}))$. For property \ref{PossCompact}, if $\mathcal{F}$ is a proper filter in $\mathsf{CO}\mathcal{RO}(UV(\mathbb{A}))$, then by the proof of Theorem \ref{MainRep}, $G=\{a\in\mathbb{A}\mid\widehat{a} \in \mathcal{F}\}$ is a proper filter in $\mathbb{A}$. Then $G$ is an element of $UV(\mathbb{A})$ and $\mathsf{CO}\mathcal{RO}(G)=\mathcal{F}$.

For part \ref{SecondThmB}, the left-to-right direction follows from part \ref{SecondThmA}. For the right-to-left direction, we will show that the map $\epsilon :x\mapsto \mathsf{CO}\mathcal{RO}(x)$ is the desired homeomorphism from $X$ to $UV(\mathsf{CO}\mathcal{RO}(X))$. To see that $\epsilon$ is injective, if ${x\neq y}$, then by $T_0$, either $x\not\leqslant y$ or $y\not\leqslant x$, which by Definition \ref{VOspace}.\ref{CloseProp} implies $\mathsf{CO}\mathcal{RO}(x)\neq \mathsf{CO}\mathcal{RO}(y)$. That $\epsilon$ is surjective follows from Definition \ref{VOspace}.\ref{PossCompact}. To see that $\epsilon$ is continuous, it suffices to show that the inverse image of each basic open is open. A basic open of $UV(\mathsf{CO}\mathcal{RO}(X))$ is $\widehat{U}$ for some $U\in \mathsf{CO}\mathcal{RO}(X)$. Then we have: 
\begin{eqnarray*}
\epsilon^{-1}[\widehat{U}]&=&\{x\in X\mid \mathsf{CO}\mathcal{RO}(x)\in \widehat{U}\} \\
&=&\{x\in X\mid U\in \mathsf{CO}\mathcal{RO}(x)\} \\
&=& \{x\in X\mid x\in U\}\\
&=& U.
\end{eqnarray*}
Finally, to see that $\epsilon^{-1}$ is continuous, we have
\begin{eqnarray*}
\epsilon[U]&=&\{\mathsf{CO}\mathcal{RO}(x)\mid x\in U\}\\
&=&\{\mathsf{CO}\mathcal{RO}(x)\mid U\in \mathsf{CO}\mathcal{RO}(x)\}\\
&=&\widehat{U}.
\end{eqnarray*}
For the last equality, the left-to-right inclusion uses that $\mathsf{CO}\mathcal{RO}(x)$ is a proper filter, while the right-to-left follows from the surjectivity of $\epsilon$.
\end{proof}

For the following, recall that for a space $X$, its specialization order is $\leqslant$.

\begin{cor}\label{UVspectral} Let $X$ be a UV-space. Then:
\begin{enumerate}
\item\label{UVspectral1} $X$ is a spectral space;
\item\label{UVspectral3} every set in $ \mathsf{CO}(X)$ is a finite union of sets from $\mathsf{CO}\mathcal{RO}(X)$;
\item\label{UVspectral4} $(X,\leqslant)$ may be obtained from a complete Heyting algebra\footnote{In Section \ref{L&H}, we strengthen `complete Heyting algebra' to `Stone locale', but we will wait to introduce this notion.} by deleting the top element, and each $U\in\mathsf{CO}\mathcal{RO}(X)$ is a filter in $(X,\leqslant)$;
\item\label{UVspectral4.5} if $X$ is finite, then $(X,\leqslant)$ may be obtained from a Boolean algebra by deleting the top element;
\item\label{UVspectral5} if $U\in\mathsf{CO}\mathcal{RO}(X)$ and $z\in X$, then there is a unique $x\in U$ and $y\in\neg U$ such that $z=x\sqcap y$ where $\sqcap$ is the meet operation in $(X,\leqslant)$.
\item\label{UVspectral6} if $U,V\in\mathsf{CO}\mathcal{RO}(X)$, then \[U\vee V=U\cup V\cup \{x\sqcap y\mid x\in U,\, y\in V \}.\]
\end{enumerate}
\end{cor}
\begin{proof} For part \ref{UVspectral1}, by Theorem \ref{SecondThm}.\ref{SecondThmB}, each UV-space $X$ is homeomorphic to the space $UV(\mathsf{CO}\mathcal{RO}(X))$, which is spectral by Proposition \ref{IsSpectral}.\ref{Spectral}. For part \ref{UVspectral3}, if $U\in \mathsf{CO}(X)$, then it is a finite union of basic open sets, so by Definition \ref{VOspace}.\ref{CloseProp}, it is a finite union of sets from $\mathsf{CO}\mathcal{RO}(X)$. 

For part \ref{UVspectral4}, as $X$ is homeomorphic to the $T_0$ space $UV(\mathsf{CO}\mathcal{RO}(X))$ of proper filters of $\mathsf{CO}\mathcal{RO}(X)$, it follows that $(X,\leqslant)$ is order-isomorphic to the poset $(UV(\mathsf{CO}\mathcal{RO}(X)),\subseteq)$ of proper filters of $\mathsf{CO}\mathcal{RO}(X)$ ordered by inclusion. As observed by Tarski \cite{Tarski1937b}, the filters of any BA (indeed, any distributive lattice) ordered by inclusion form a complete Heyting algebra, so the proper filters ordered by inclusion form a complete Heyting algebra minus the top element. Finally, suppose $U=\widehat{a}$ for $a\in \mathbb{A}$, and $F,G\in UV(\mathbb{A})$ are such that $F,G\in \widehat{a}$. Then $a\in F\cap G=F\sqcap G$, so $F\sqcap G\in \widehat{a}=U$. It follows, given that $U$ is an upset, that $U$ is a filter in $(UV(\mathbb{A}),\subseteq)$.

For part \ref{UVspectral4.5}, if $X$ is finite, then the BA $\mathsf{CO}\mathcal{RO}(X)$ is finite. As in part \ref{UVspectral4}, ${(X,\leqslant)}$ is order-isomorphic to the poset of proper filters of $\mathsf{CO}\mathcal{RO}(X)$ ordered by inclusion. Since any filter in a finite BA is principal, we obtain that $(X,\leqslant)$ is order-isomorphic to the poset of proper \textit{principal} filters of $\mathsf{CO}\mathcal{RO}(X)$ ordered by inclusion, which is obviously isomorphic to $\mathsf{CO}\mathcal{RO}(X)$ minus its top element.

For part \ref{UVspectral5}, let $X=UV(\mathbb{A})$. If $U\in \mathsf{CO}\mathcal{RO}(UV(\mathbb{A}))$, then by Theorem \ref{SecondThm}.\ref{SecondThmB} and the proof of Theorem \ref{MainRep}, we have $U=\widehat{a}$ and $\neg U=\widehat{-a}$ for some $a\in\mathbb{A}$, which implies $\mathord{\uparrow}a\in U$ and $\mathord{\uparrow}\mathord{-a}\in \neg U$. Let $\sqcap$ and $\sqcup$ be the meet and join operations in the  Heyting algebra arising from $(UV(\mathbb{A}),\subseteq)$, i.e., $F\sqcap G=F\cap G$ and $F\sqcup G$ is the filter generated by $F\cup G$. Let $\top$ be the top element of the Heyting algebra, which we may identify with the improper filter in $\mathbb{A}$. Thus, $\mathord{\uparrow}a\sqcup\mathord{\uparrow}\mathord{-}a=\top$. Now for any $F\in UV(\mathbb{A})$, we have $F=(F\sqcup \mathord{\uparrow}a)\sqcap(F\sqcup\mathord{\uparrow}\mathord{-}a)$. Suppose $G\in U$ and $H\in \neg U$, which implies $\mathord{\uparrow}a\subseteq G$ and  $\mathord{\uparrow}\mathord{-a}\subseteq H$, and $F=G\sqcap H$. Then we have
\[ F\sqcup \mathord{\uparrow}a =(G\sqcap H)\sqcup \mathord{\uparrow}a= (G\sqcup \mathord{\uparrow}a)\sqcap (H\sqcup \mathord{\uparrow}a)=G\sqcap \top = G,\]
and similarly $F\sqcup\mathord{\uparrow}\mathord{-}a=H$. This completes the proof of part \ref{UVspectral5}.

For part \ref{UVspectral6}, we show that $\widehat{a}\vee \widehat{b}= \widehat{a}\cup \widehat{b}\cup \{F\sqcap G\mid F\in \widehat{a},\, G\in \widehat{b} \}$. By the proof of Theorem \ref{MainRep}, $\widehat{a}\vee \widehat{b}=\widehat{a\vee b}$. To see that $\widehat{a\vee b}\supseteq \widehat{a}\cup \widehat{b}\cup \{F\sqcap G\mid F\in \widehat{a},\, G\in \widehat{b} \}$, obviously $\widehat{a\vee b}\supseteq \widehat{a}\cup \widehat{b}$. If $F\in \widehat{a}$ and $G\in \widehat{b}$, so $a\in F$ and $b\in G$, then $a\vee b\in F\cap G=F\sqcap G$, so $F\sqcap G\in \widehat{a\vee b}$. To see that $\widehat{a\vee b}\subseteq \widehat{a}\cup \widehat{b}\cup \{F\sqcap G\mid F\in \widehat{a},\, G\in \widehat{b} \}$, if $H\in \widehat{a\vee b}$, so $a\vee b\in H$, and $H\not\in \widehat{a}\cup \widehat{b}$, so $a\not\in H$ and $b\not\in H$, then we claim that $H= (H\sqcup \mathord{\uparrow}a)\sqcap (H\sqcup \mathord{\uparrow}b)$. For if $c$ is in the right-hand side, then there are $a_0\in H$ and $b_0\in H$ such that $a_0\wedge a\leq c$ and $b_0\wedge b\leq c$, which implies $a_0\wedge b_0\wedge (a\vee b)\leq c$. Then since $a_0,b_0,a\vee b\in H$, we have $c\in H$. Finally, both $H\sqcup \mathord{\uparrow}a$ and $H\sqcup \mathord{\uparrow}b$ are proper filters. For if $H\sqcup \mathord{\uparrow}a$ is improper, then $-a\in H$, which with $a\vee b\in H$ implies $b\in H$, which contradicts what we derived above. Similarly, that $H\sqcup \mathord{\uparrow}b$ is improper leads to a contradiction.\end{proof}

\begin{cor}\label{StoneCor} For any Stone space $X$, $\mathscr{UV}(X)$ is a UV-space.
\end{cor}
\begin{proof} By Proposition \ref{UVStone}, for any Stone space $X$, $\mathscr{UV}(X)$ is homeomorphic to $UV(\mathsf{Clop}(X))$, which is a UV-space by Theorem \ref{SecondThm}.\ref{SecondThmA}.
\end{proof}

\section{Morphisms and choice-free duality for BAs}\label{DualitySection}

To go beyond representation to categorical duality, we introduce appropriate morphisms. A \textit{spectral map} \cite{Hochster1969} between spectral spaces $X$ and $X'$ is a map $f\colon X\to X'$ such that $f^{-1}[U]\in\mathsf{CO}(X)$ for each $U\in\mathsf{CO}(X')$, which implies that $f$ is continuous. We combine this definition with the standard notion (in modal logic) of a {p-morphism} between ordered sets (see, e.g., \cite[p.~30]{Chagrov1997}).

\begin{definition}\label{UVmapDef} A \textit{UV-map} between UV-spaces $X$ and $X'$ is a spectral map $f\colon X\to X'$ that also satisfies the \textit{p-morphism condition}:
\[\mbox{if $f(x) \leqslant' y'$, then $\exists y: x\leqslant y$ and $f(y)=y'$}.\]
\end{definition}

\begin{figure}[h]
\begin{center}
\begin{tikzpicture}[->,>=stealth',shorten >=1pt,shorten <=1pt, auto,node
distance=2cm,thick,every loop/.style={<-,shorten <=1pt}] \tikzstyle{every state}=[fill=gray!20,draw=none,text=black]

\node[circle,draw=black!100,fill=black!100, label=below:$x$,inner sep=0pt,minimum size=.175cm] (x) at (0,0) {{}};

\draw[rotate=90] (1,0) ellipse (2cm and .7cm);

\draw[rotate=90] (1,-3) ellipse (2cm and .7cm);

\draw[rotate=90] (1,-6) ellipse (2cm and .7cm);

\draw[rotate=90] (1,-9) ellipse (2cm and .7cm);

\node[circle,draw=black!100,fill=black!100, label=below:$\;f(x)$,inner sep=0pt,minimum size=.175cm] (x') at (3,0) {{}};

\node[circle,draw=black!100,fill=black!100, label=above:$y'$,inner sep=0pt,minimum size=.175cm] (y') at (3,2) {{}};

\path (x') edge[->] node {{}} (y'); 

\path (x) edge[bend left,dotted,->] node {{}} (x');

\node[circle,draw=black!100,fill=black!100, label=below:$x$,inner sep=0pt,minimum size=.175cm] (x2) at (6,0) {{}};

\node[circle,draw=black!100,fill=black!100, label=above:$\exists y$,inner sep=0pt,minimum size=.175cm] (y2) at (6,2) {{}};

\node[circle,draw=black!100,fill=black!100, label=below:$\;f(x)$,inner sep=0pt,minimum size=.175cm] (x'2) at (9,0) {{}};

\node[circle,draw=black!100,fill=black!100, label=above:$y'$,inner sep=0pt,minimum size=.175cm] (y'2) at (9,2) {{}};

\path (x2) edge[->] node {{}} (y2); 
\path (x'2) edge[->] node {{}} (y'2); 

\node at (4.5,1) {{$\Rightarrow$}};

\path (x2) edge[bend left,dotted,->] node {{}} (x'2); 
\path (y2) edge[bend left,dotted,->] node {{}} (y'2); 

\end{tikzpicture}
\end{center}
\caption{The p-morphism condition of UV-maps.}\label{p-morphismFig}
\end{figure}
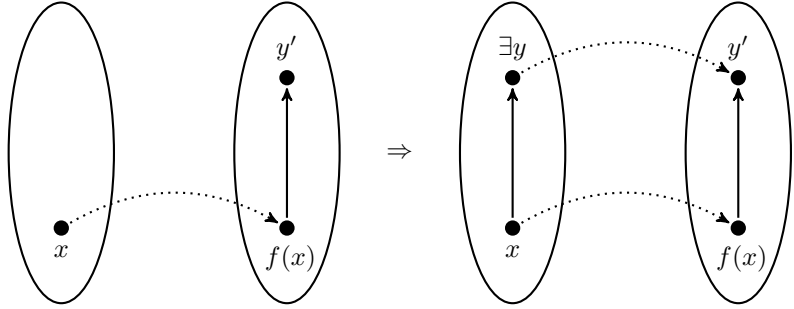

\begin{remark} A UV-map, like any continuous map, preserves the specialization order: if $x\leqslant y$, then $f(x)\leqslant f(y)$.
\end{remark}

\begin{fact}\label{PosetMaps} Let $P$ and $P'$ be partial orders, and let $f\colon P\to P'$ be an order-preserving map satisfying the p-morphism condition. If $U\in\mathcal{RO}(P')$, then $f^{-1}[U]\in\mathcal{RO}(P)$ \textnormal{(}where we regard $P,P'$ as spaces given by their upset topologies\textnormal{)}.
\end{fact}

\begin{proof} To see that $f^{-1}[U]\in\mathsf{Up}(P)$, suppose $x\in f^{-1}[U]$ and $x\leqslant y$. Then $f(x)\in U$, and since $f$ is order-preserving, $f(x)\leqslant' f(y)$, so $U\in \mathsf{Up}(P')$ implies $f(y)\in U$ and hence $y\in f^{-1}[U]$. Now to see that $f^{-1}[U]\in\mathcal{RO}(P)$, suppose $x\not\in f^{-1}[U]$, so $f(x)\not\in U$. Then since $U\in\mathcal{RO}(P')$, there is a $y'\geqslant' f(x)$ such that for all $z'\geqslant' y'$, we have $z'\not\in U$. It follows by the p-morphism condition that there is a $y$ such that $x\leqslant y$ and $f(y)=y'$. Then for any $z$ such that $y\leqslant z$, we have $f(y)\leqslant' f(z)$ and hence $y'\leqslant' f(z)$, which implies $f(z)\not\in U$ by our reasoning above, so $z\not\in f^{-1}[U]$. Thus, we have shown that if $x\not\in f^{-1}[U]$, then there is a $y\geqslant x$ such that for all $z\geqslant y$, $z\not\in f^{-1}[U]$. By (\ref{ROeq}), this completes the proof that $f^{-1}[U]\in\mathcal{RO}(P)$.
\end{proof}

From Fact \ref{PosetMaps} and the definition of UV-maps as special spectral spaces, we have the following.

\begin{cor}\label{InverseCor} Let $X$ and $X'$ be UV-spaces and $f\colon X\to X'$ a UV-map. Then $f^{-1}[U]\in\mathsf{CO}\mathcal{RO}(X)$ for every $U\in\mathsf{CO}\mathcal{RO}(X')$.\footnote{Cf.~the notion of an \textit{R-map} in \cite{Carnahan1973}, which is a map between spaces such that the inverse image of each regular open set is regular open.}
\end{cor}

Conversely, the condition that the inverse image of a $\mathsf{CO}\mathcal{RO}$ set is also $\mathsf{CO}\mathcal{RO}$ (or simply $\mathsf{CO}$) implies that $f$ is a spectral map.

\begin{fact}\label{IsSpectralMap} Let $X$ and $X'$ be UV-spaces. If $f\colon X\to X'$ is such that $f^{-1}[U]\in\mathsf{CO}(X)$ for every $U\in\mathsf{CO}\mathcal{RO}(X')$, then $f$ is a spectral map.
\end{fact}

\begin{proof} Suppose $f\colon X\to X'$ satisfies the assumption, and $U\in\mathsf{CO}(X')$. By Proposition \ref{UVspectral}.\ref{UVspectral3}, $U$ is a finite union $\underset{i\in I}{\bigcup} U_i$ of sets $U_i\in\mathsf{CO}\mathcal{RO}(X')$. Then $f^{-1}[U]=f^{-1}[\underset{i\in I}{\bigcup} U_i]=\underset{i\in I}{\bigcup} f^{-1}[U_i]$. By the assumption,  $f^{-1}[U_i]\in\mathsf{CO}(X)$, so $f^{-1}[U]$ is a finite union of compact opens and is therefore compact open. Thus, $f$ is a spectral map.
\end{proof}

The following simple lemma is also useful.

\begin{lemma}\label{SubbasisLem} Let $X$ and $Y$ be spectral spaces and $f: X\to Y$. If for each set $U$ in some subbasis for $Y$, we have $f^{-1}[U]\in\mathsf{CO}(X)$, then $f$ is a spectral map.
\end{lemma}

\begin{proof} By definition, every open set is a union of finite intersections of subbasic sets. Thus, every compact open set $V$ is a finite union $V_1\cup\dots\cup V_n$ of finite intersections of subbasic sets. Then since 
\[f^{-1}[V]=f^{-1}[V_1\cup\dots\cup V_n]=f^{-1}[V_1]\cup \dots\cup f^{-1}[V_n],\]
we have that $f^{-1}[V]$ is compact open if each $f^{-1}[V_i]$ is compact open. Now each $V_i$ is $U_1\cap\dots\cap U_n$ for some subbasic sets $U_1,\dots,U_n$. Then since
\[f^{-1}[V_i]=f^{-1}[U_1\cap\dots\cap U_n]=f^{-1}[U_1]\cap\dots\cap f^{-1}[U_n],\]
we have that $f^{-1}[V_i]$ is compact open if each $f^{-1}[U_j]$ is compact open. By assumption, each $f^{-1}[U_j]$ is compact open, so we are done.\end{proof}

One can easily check that UV-spaces with UV-maps form a category. We now prove the promised categorical duality.

\begin{theorem}\label{DualityThm} The category of UV-spaces with UV-maps is dually equivalent to the category of Boolean algebras with Boolean homomorphisms.
\end{theorem}

\begin{proof} Suppose $h\colon \mathbb{A}\to\mathbb{B}$ is a BA homomorphism. Given $F\in UV(\mathbb{B})$, let $h_+(F)=h^{-1}[F]$. Then since $h$ is a homomorphism, and $F$ is a proper filter in $\mathbb{B}$, it follows that $h_+(F)$ is a proper filter in $\mathbb{A}$. Thus, \[h_+:UV(\mathbb{B})\to UV(\mathbb{A}).\] We claim that $h_+$ is a UV-map. First, to see that $h_+$ is a spectral map, it suffices by Lemma \ref{SubbasisLem} to show that for each basic open $\widehat{a}$ of $UV(\mathbb{A})$, we have $h_+^{-1}[\widehat{a}]\in\mathsf{CO}(UV(\mathbb{B}))$. Indeed, 
\begin{eqnarray*}
h_+^{-1}[\widehat{a}]&=&\{F\in UV(\mathbb{B})\mid h_+(F)\in \widehat{a}\} \\
&=&\{F\in UV(\mathbb{B})\mid h^{-1}[F]\in \widehat{a}\} \\
&=& \{F\in UV(\mathbb{B})\mid a\in h^{-1}[F]\} \\
&=& \{F\in UV(\mathbb{B})\mid h(a)\in F\}\\
&=&\widehat{h(a)},
\end{eqnarray*} 
and $\widehat{h(a)}$ is compact open by the proof of Proposition \ref{IsSpectral}.\ref{Spectral}. 

Next, we show that $h_+$ satisfies the p-morphism condition:
\[\mbox{if $h_+(F) \leqslant' G'$, then $\exists G: F\leqslant G$ and $h_+(G)=G'$}.\]
If $G'\in UV(\mathbb{A})$ and $h_+(F)\subseteq G'$, we claim that the filter $G$ generated by $h[G']\cup F$ is a proper filter. If not, then there are some $c_1,\dots,c_n\in h[G']$  such that $-(c_1\wedge\dots\wedge c_n)\in F$. Since $c_1,\dots,c_n\in h[G']$, there are some $c'_1,\dots,c'_n\in G'$ such that $h(c'_i)=c_i$, so $-(h(c'_1)\wedge\dots\wedge h(c'_n))\in F$. Then since $h$ is a homomorphism, we have $h(-(c_1\wedge\dots\wedge c_n))\in F$, so that $-(c_1\wedge\dots\wedge c_n)\in h^{-1}[F]=h_+(F)$, which with $h_+(F)\subseteq G'$ implies $-(c_1\wedge\dots\wedge c_n)\in G'$, which contradicts the fact that $c'_1,\dots,c'_n\in G'$ and $G'$ is a proper filter. Thus, $G$ is indeed a proper filter, and we have both $F\subseteq G$ and $G'\subseteq h^{-1}[G]=h_+(G)$. Finally, we claim that $h_+(G)\subseteq G'$.\footnote{Thanks to David Gabelaia and Mamuka Jibladze for pointing out this strengthening of the original proof.} For if $c'\in h_+(G)$, so $h(c')\in G$, then by  definition of $G$ there is a $b'\in G'$ and $a\in F$ such that $h(b')\wedge a\leq h(c')$, which implies $a\leq -h(b')\vee h(c')$ and hence $a\leq h(-b'\vee c')$. Then since $a\in F$, we have $h(-b'\vee c')\in F$, so $-b'\vee c'\in h^{-1}[F]=h_+(F)$.  Since $h_+( F )\subseteq G'$, it follows that $-b'\vee c'\in G'$, which with $b'\in G'$ implies $c'\in G'$, which completes the proof that $h_+(G)\subseteq G'$. Thus, $h_+(G)= G'$, so $h_+$ satisfies the p-morphism condition.

Finally, it is easy to see that $(\cdot)_+$ preserves the identity and composition. Thus, together $UV(\cdot)$ and $(\cdot)_+$ give us a contravariant functor from the category of BAs with BA homomorphisms to the category of UV-spaces with UV-maps. 

In the other direction, suppose $f: X\to Y$ is a UV-map. Given $U\in\mathsf{CO}\mathcal{RO}(Y)$, let $f^+(Y)=f^{-1}[Y]$. Then by Corollary \ref{InverseCor}, \[f^+: \mathsf{CO}\mathcal{RO}(Y)\to \mathsf{CO}\mathcal{RO}(X).\] We claim that $f^+$ is a BA homomorphism. First, $f^+(U\wedge V)=f^{-1}[U\cap V]=f^{-1}[U]\cap f^{-1}[V]=f^+(U)\wedge f^+(V)$. Second, since $f$ is a UV-map, we have that for all $x\in X$ and $U \in \mathsf{CO}\mathcal{RO}(Y)$, $\mathord{\Uparrow}f(x)\cap U=\varnothing$ iff ${\mathord{\Uparrow}x\cap f^{-1}[U]=\varnothing}$. It follows that $f^{-1}[\mathsf{int}_\leqslant (Y\setminus U)]=\mathsf{int}_\leqslant (X\setminus f^{-1}[U])$ and hence $f^+(\neg U)=\neg f^+(U)$. It is also easy to see that $(\cdot)^+$ preserves the identity and composition. Thus, together $\mathsf{CO}\mathcal{RO}(\cdot)$ and $(\cdot)^+$ give us a contravariant functor from the category of UV-spaces with UV-maps to the category of BAs with BA homomorphisms.

In Theorems \ref{MainRep} and \ref{SecondThm}.\ref{SecondThmB} we showed that each BA $\mathbb{A}$ is isomorphic to $\mathsf{CO}\mathcal{RO}(UV(\mathbb{A}))$ and each UV-space $X$ is homeomorphic to $UV(\mathsf{CO}\mathcal{RO}(X))$.

Finally, it is not difficult to check that the following diagrams commute for any BA homomorphism $h:\mathbb{A}\to\mathbb{B}$ and UV-map $f:X\to Y$:
\begin{center}
\begin{tikzpicture}[->,>=stealth',shorten >=1pt,shorten <=1pt, auto,node
distance=2cm,thick,every loop/.style={<-,shorten <=1pt}]
\tikzstyle{every state}=[fill=gray!20,draw=none,text=black]

\node (A) at (0,2) {{$\mathbb{A}$}};
\node (A') at (0,0) {{$\mathsf{CO}\mathcal{RO}(UV(\mathbb{A}))$}};
\node (B) at (5,2) {{$\mathbb{B}$}};
\node (B') at (5,0) {{$\mathsf{CO}\mathcal{RO}(UV(\mathbb{B}))$}};

\path (A) edge[->] node {{$h$}} (B);
\path (A') edge[<-] node {{}} (A);
\path (B) edge[->] node {{}} (B');
\path (B') edge[<-] node {{$(h_+)^+$}} (A');

\end{tikzpicture}
\end{center}

\begin{center}
\begin{tikzpicture}[->,>=stealth',shorten >=1pt,shorten <=1pt, auto,node
distance=2cm,thick,every loop/.style={<-,shorten <=1pt}]
\tikzstyle{every state}=[fill=gray!20,draw=none,text=black]

\node (A) at (0,2) {{$X$}};
\node (A') at (0,0) {{$UV(\mathsf{CO}\mathcal{RO}(X))$}};
\node (B) at (5,2) {{$Y$}};
\node (B') at (5,0) {{$UV(\mathsf{CO}\mathcal{RO}(Y))$}};

\path (A) edge[->] node {{$f$}} (B);
\path (A') edge[<-] node {{}} (A);
\path (B) edge[->] node {{}} (B');
\path (B') edge[<-] node {{$(f^+)_+$}} (A');

\end{tikzpicture}
\end{center}
This completes the proof.
\end{proof}

\section{The hyperspace approach and the localic approach}\label{L&H} In this section, we relate the hyperspace approach to choice-free duality using UV-spaces to the localic approach using Stone locales.

Recall that a \textit{locale} is a complete lattice $L$ satisfying the join-infinite distributive law for each $a\in L$ and $Y\subseteq L$:
\[a\wedge \bigvee Y= \bigvee \{a\wedge y\mid y\in Y\}.\]
The collection of open sets of any space ordered by $\subseteq$ is a locale. In point-free topology, it is locales rather than spaces that are the basic objects. If we ignore choices of signature, then a lattice is a locale iff it is a complete Heyting algebra. For more information on locales, see, e.g., \cite{Johnstone1982,Picado2012}.

A locale is \textit{compact} if $\bigvee Y=\one$ implies $\bigvee Y_0=\one$ for some finite $Y_0\subseteq Y$. A locale is \textit{zero-dimensional} if each element of the locale is a join of complemented elements, where an element $a$ is complemented if there exists an element $b$ such that $a\wedge b=\zero$ and $a\vee b=\one$.

\begin{definition} A \textit{Stone locale} is a compact zero-dimensional locale.
\end{definition}

The name `Stone locale' is justified by the fact that the locale of any Stone space is a Stone locale, and assuming the Boolean Prime Ideal Theorem, every Stone locale $L$ is the locale of opens of a Stone space, namely the Stone dual of the BA of complemented elements of $L$.

As mentioned in Section \ref{PossSection}, Stone locales provide another kind of choice-free Stone duality for BAs. A proof of the following may be found in \cite{Bezhanishvili2015}.

\begin{theorem}\label{StoneLocaleThm} The category of Stone locales with localic maps\footnote{For the definition of localic maps, see, e.g., \cite[\S~II.2]{Picado2012}.} is dually equivalent to the category of BAs with Boolean homomorphisms.
\end{theorem}

The key to Theorem \ref{StoneLocaleThm} is the following correspondence.

\begin{lemma}\label{ZLem} $\,$
\begin{enumerate}
\item\label{ZLem1} For any BA $\mathbb{A}$, $(\mathrm{Filt}(\mathbb{A}),\subseteq)$ is a Stone locale.
\item\label{ZLem2} $L$ is a Stone locale iff $L$ is isomorphic to $(\mathrm{Filt}(Z(L)),\subseteq)$ where $Z(L)$ is the Boolean algebra of complemented elements of $L$.
\end{enumerate}
\end{lemma}
\begin{proof}(sketch) Part \ref{ZLem1} is straightforward to check. For part \ref{ZLem2}, the right-to-left direction follows from part \ref{ZLem1}. For the left-to-right direction, the isomorphism sends $a\in L$ to $\mathord{\uparrow}a\cap Z(L)$ (see \cite{Bezhanishvili2015}).
\end{proof}

We can characterize UV-spaces as the result of putting an appropriate topology on the (non-maximum) elements of a Stone locale. Given a Stone locale $L$, just as Johnstone \cite[\S~4.1]{Johnstone1982} defines the \textit{Vietoris space of $L$},  we may define the \textit{upper Vietoris space of $L$}. The starting observation is that in defining the upper Vietoris space of a Stone space $X$, instead of taking the points of the new space to be the nonempty closed sets of $X$, we can take the points to be the complements of such sets, i.e., the open sets of $X$ not equal to $X$. Then for $U\in\Omega(X)$, instead of defining 
\[\Box U =\{F\in \mathsf{F}(X)\mid F\subseteq U\},\]
we define
\begin{eqnarray*}
\blacksquare U &=&\{V\in \Omega(X)\setminus\{X\}\mid V^c\subseteq U\} \\
& =&\{V\in \Omega(X)\setminus\{X\}\mid U\cup V=X\}
\end{eqnarray*}
and let the topology be generated by $\{\blacksquare U\mid U\in\Omega(X)\}$. With this change of perspective, we can define the Vietoris space entirely in terms of the locale $\Omega(X)$, motivating the following definition.

 \begin{definition}\label{UVofLocale} The \textit{upper Vietoris space} of a Stone locale $L$ is the space whose set of points is $L^-=\{x\in L\mid x\neq 1\}$ and whose topology is generated by the sets \[\blacksquare x= \{y\in L^-\mid x\vee y=1\},\; x\in L.\]
 \end{definition}

Now suppose $L$ is the Stone locale $(\mathrm{Filt}(\mathbb{A}),\subseteq)$ for a BA $\mathbb{A}$. The join $F\vee G$ of two filters $F,G\in L$ is the filter generated by $F\cup G$, and the top element $1$ of $L$ is the improper filter. Our $UV(\mathbb{A})$ is exactly the topological space based on $L^-$ with the topology generated by the sets
\[\widehat{a}=\{F\in \mathrm{PropFilt}(\mathbb{A})\mid a\in F\},\;a\in\mathbb{A}.\]
We can now see that $UV(\mathbb{A})$ is exactly the upper Vietoris space of the Stone locale $(\mathrm{Filt}(\mathbb{A}),\subseteq)$.

 \begin{prop}\label{SameTopologies} Let $L$ be the Stone locale of filters of a BA $\mathbb{A}$. Then the topology on $L^-$ generated by $\{\blacksquare x\mid x\in L\}$ is equal to the topology on $L^-$ generated by $\{\widehat{a}\mid a\in\mathbb{A}\}$.\end{prop}
 
 \begin{proof} Given $a\in\mathbb{A}$, we have: \[\widehat{a}=\blacksquare\mathord{\uparrow}\mathord{-}a.\] For $\widehat{a}\subseteq\blacksquare\mathord{\uparrow}\mathord{-}a$, if $F\in\widehat{a}$, so $F$ is a proper filter with $a\in F$, then clearly $F\vee \mathord{\uparrow}\mathord{-}a$, i.e., the filter generated by $F\cup \mathord{\uparrow}\mathord{-}a$, is the improper filter, so $F\in \blacksquare\mathord{\uparrow}\mathord{-}a$. For $\widehat{a}\supseteq\blacksquare\mathord{\uparrow}\mathord{-}a$, if $F\in \blacksquare\mathord{\uparrow}\mathord{-}a$, so $F$ is a proper filter such that the filter generated by $F\cup \mathord{\uparrow}\mathord{-}a$ is improper, then $a\in F$, so $F\in\widehat{a}$.
 
 Given $F\in L$, we have:
 \[\blacksquare F=\bigcup\{\widehat{-a}\mid a\in F\}.\]
 For the left-to-right inclusion, suppose $G\in \blacksquare F$, so $G$ is a proper filter such that $F\vee G$ is the improper filter. Hence there is some $a\in F$ such that $-a\in G$, so that $G\in\widehat{-a}$ and hence $G\in \bigcup\{\widehat{-a}\mid a\in F\}$. From right to left, suppose $G\in \bigcup\{\widehat{-a}\mid a\in F\}$, so for some $a\in F$, $G\in\widehat{-a}$, which means $-a\in G$. Then clearly $F\vee G$ is the improper filter, so $G\in \blacksquare F$.
 \end{proof}
 
Combining Proposition \ref{SameTopologies} with Definitions \ref{UVofBA} and \ref{UVofLocale}, we have the following as an immediate corollary.
 
 \begin{cor}\label{SameTopCor} For any BA $\mathbb{A}$, $UV(\mathbb{A})$ is the upper Vietoris space of the Stone locale $(\mathrm{Filt}(\mathbb{A}),\subseteq)$.
 \end{cor}
 
We can now justify our choice of the terminology `UV-space' with the following choice-free characterization.
 
 \begin{theorem}\label{UVJustification} $X$ is a UV-space iff $X$ is homeomorphic to the upper Vietoris space of a Stone locale.
 \end{theorem}
 
 \begin{proof} Suppose $X$ is a UV-space. Then by Theorem \ref{SecondThm}.\ref{SecondThmB}, $X$ is homeomorphic to $UV(\mathsf{CO}\mathcal{RO}(X))$. By Corollary \ref{SameTopCor}, $UV(\mathsf{CO}\mathcal{RO}(X))$ is the upper Vietoris space of the Stone locale $(\mathrm{Filt}(\mathsf{CO}\mathcal{RO}(X)),\subseteq)$. Thus, $X$ is homeomorphic to the upper Vietoris space of a Stone locale. 
 
 Conversely, suppose $X$ is homeomorphic to the upper Vietoris space of a Stone locale $L$. By Lemma \ref{ZLem}.\ref{ZLem2}, $L$ is isomorphic to $(\mathrm{Filt}(Z(L)),\subseteq)$. Thus, $X$ is homeomorphic to the upper Vietoris space of $(\mathrm{Filt}(Z(L)),\subseteq)$, which is equal to $UV(Z(L))$ by Corollary \ref{SameTopCor}, which is a UV-space by Theorem \ref{SecondThm}.\ref{SecondThmA}. Thus, $X$ is a UV-space.
 \end{proof}
 
 Theorem \ref{UVJustification} is a choice-free point-free analogue of the statement that $X$ is a UV-space iff $X$ is homeomorphic to $\mathscr{UV}(Y)$ for a Stone space $Y$. The left-to-right direction of that statement assumes the Boolean Prime Ideal Theorem (see Section \ref{UVStoneSection}). But by switching from Stone spaces to Stone locales, one obtains Theorem \ref{UVJustification} without choice.

\begin{remark}For a Stone locale $L$, in addition to defining the Vietoris space of $L$, Johnstone \cite[\S~4.1]{Johnstone1982} defines the \textit{Vietoris locale of $L$}, also known as the \textit{Vietoris powerlocale of $L$}.\footnote{Johnstone studies these constructions for any compact regular locale $L$, but here we need only consider Stone locales.} This is a purely localic construction, and the terminology is justified by the fact that the space of points of the Vietoris locale of $L$ is homeomorphic to the Vietoris space of $L$. Similarly, one can give a purely localic construction of the \textit{upper Vietoris locale of $L$}, also known as the \textit{upper powerlocale of $L$} \cite{Vickers1997,Vickers2009}, such that its space of points is homeomorphic to the upper Vietoris space of $L$.\end{remark}

Figure \ref{Locales&Spaces} below relates the different constructions we have discussed, viewed as ways of constructing the dual UV-space of a given BA.

 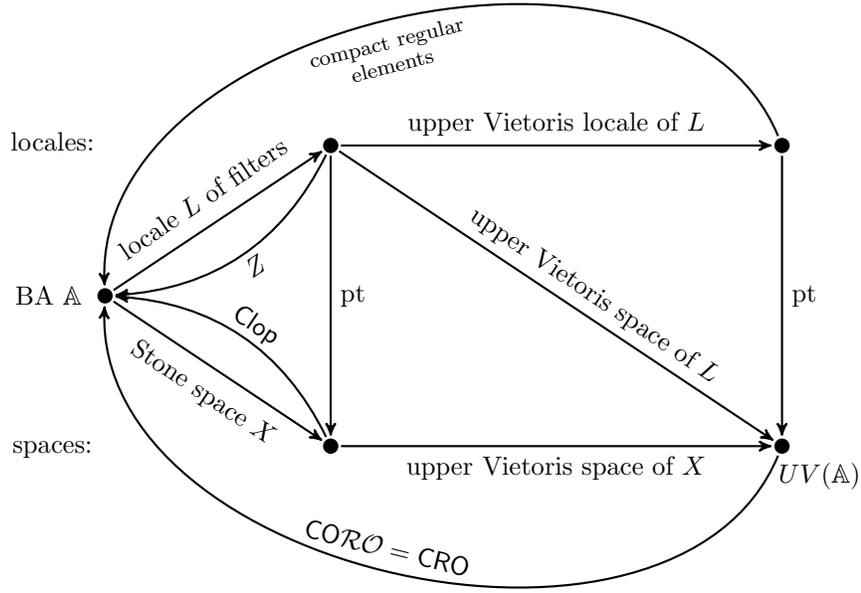
\begin{figure}[h]
\begin{center}
\begin{tikzpicture}[->,>=stealth',shorten >=1pt,shorten <=1pt, auto,node
distance=2cm,thick,every loop/.style={<-,shorten <=1pt}] \tikzstyle{every state}=[fill=gray!20,draw=none,text=black]

\node [circle,draw=black!100,fill=black!100, label=below:$$,inner sep=0pt,minimum size=.175cm]  (a) at (0,0) {{}}; 
\node [circle,draw=black!100,fill=black!100, label=above:$$,inner sep=0pt,minimum size=.175cm] (b) at (3,2) {{}}; 
\node  [circle,draw=black!100,fill=black!100, label=above:$$,inner sep=0pt,minimum size=.175cm] (b') at (9,2) {{}}; 
\node [circle,draw=black!100,fill=black!100, label=above:$$,inner sep=0pt,minimum size=.175cm]  (c) at (3,-2) {{}}; 
\node  [circle,draw=black!100,fill=black!100, label=above:$$,inner sep=0pt,minimum size=.175cm] (c') at (9,-2) {{}};

\path (a) edge[->] node {{}} (b); 
\node at (1.3,1.3) {{$\rotatebox{34}{locale $L$ of filters}$}};

\path (b) edge[->, bend left] node {{}} (a); 
\node at (2,0.4) {{$\rotatebox{34}{$Z$}$}};

\path (a) edge[->] node {{}} (c);
\node at (1.3,-1.3) {{$\rotatebox{-34}{Stone space $X$}$}};

\path (c) edge[->, bend right] node {{}} (a);
\node at (2,-0.4) {{$\rotatebox{-34}{$\mathsf{Clop}$}$}};

\path (b) edge[->] node {{upper Vietoris locale of $L$}} (b'); 

\path (c') edge[<-] node {{upper Vietoris space of $X$}} (c); 
\path (b) edge[->] node {{pt}} (c); 
\path (b') edge[->] node {{pt}} (c'); 
\path (b) edge[->] node {{}} (c'); 

\path (b') edge[->,bend right=80] node {{}} (a);
\node at (3.75,3.35) {{$\rotatebox{14}{{\footnotesize compact regular}}$}};
\node at (3.75,3.05) {{$\rotatebox{14}{{\footnotesize \, elements}}$}};

\path (c') edge[->,bend left=80] node {{}} (a);
\node at (3.75,-3.35) {{$\rotatebox{-14}{$\mathsf{CO}\mathcal{RO}=\mathsf{CRO}$}$}};

\node at (6.5,0) {{$\rotatebox{-34}{upper Vietoris space of $L$}$}};

\node at (-.75,0) {{BA $\mathbb{A}$}};
\node at (9.5,-2.4) {{$UV(\mathbb{A})$}};

\node at (-.7,2.05) {{locales:}};
\node at (-.7,-2.05) {{spaces:}};

\end{tikzpicture}
\caption{Routes to the dual UV-space of a BA and back.}\label{Locales&Spaces}
\end{center}
\end{figure}

\section{Duality dictionary}\label{DictionarySection} In this section, we explain the dictionary in Table \ref{Dictionary} for translating between BA notions and UV notions.

\begin{table}[h]
 \begin{center}
 \scriptsize
 \begin{tabular}{l|l|l}
   \hline
   {\bf BA} & {\bf UV} & {\bf Stone } \\
   \hline \hline
    BA & UV-space & Stone space \\
    \hline
   homomorphism & UV-map & continuous map \\
   \hline
   filter & $\mathord{\Uparrow}x$, $x\in X$ & closed set   \\
   \hline
   ideal & $U\in \mathsf{O}\mathcal{RO}(X)$ & open set\\
   \hline
    principal filter & $U\in \mathsf{CO}\mathcal{RO}(X)$ & clopen set    \\
       \hline
    principal ideal & $U\in \mathsf{CO}\mathcal{RO}(X)$ & clopen set   \\
   \hline
   maximal filter & $\{x\}$, $x\in\mathsf{Max}_\leqslant(X)$  & $\{x\}$, $x\in X$ \\
   \hline
   maximal ideal & $X\setminus\mathord{\Downarrow}x$, $x\in\mathsf{Max}_\leqslant(X)$ & $X\setminus \{x\}$, $x\in X$ \\
      \hline
   normal ideal & $U\in\mathsf{RO}(X)$ &  $U\in\mathsf{RO}(X)$ \\
   \hline
   relativization & subspace $U\in\mathsf{CO}\mathcal{RO}(X)$ & subspace $U\in\mathrm{Clop}(X)$ \\
   \hline
     complete algebra & complete UV-space & ED Stone space  \\
     \hline
   atom & isolated point & isolated point \\
   \hline
            atomless algebra & $X_\mathrm{iso}=\varnothing$ & $X_\mathrm{iso}=\varnothing$ \\
            \hline
      atomic algebra & $\mathsf{cl}(X_\mathrm{iso})=X$ & $\mathsf{cl}(X_\mathrm{iso})=X$ \\
   \hline
   homomorphic image  & subspace induced by $\mathord{\Uparrow}x$, $x\in X$ & subspace induced by closed set  \\
   \hline
   subalgebra & image under UV-map  & image under continuous map  \\
   \hline
    direct product & UV-sum  & disjoint union  \\
   \hline
   canonical extension & $\mathcal{RO}(X)$ & $\wp(X)$ \\
   \hline
   MacNeille completion & $\mathsf{RO}(X)$ & $\mathsf{RO}(X)$ \\

   \hline
 \end{tabular}
 \end{center}
\vspace{5mm} \caption{Dictionary for \textbf{BA}, \textbf{UV}, and \textbf{Stone}.}\label{Dictionary}
\end{table}

\subsection{Filters and ideals}

For a filter $F$ and ideal $I$ in a BA $\mathbb{A}$, we define:
\begin{eqnarray*}
\eta (F)&=& \bigcap\{\widehat{a}\mid a\in F\};\\
\zeta(I)&=& \bigcup\{\widehat{a}\mid a\in I\}.
\end{eqnarray*}

\begin{fact}\label{EtaFact} Let $\mathbb{A}$ be a BA and $X$ its dual UV-space. The map $\eta$ is a dual isomorphism between the poset of proper filters of $\mathbb{A}$ \textnormal{(}ordered by inclusion\textnormal{)} and the poset of principal upsets in the specialization order of $X$ \textnormal{(}ordered by inclusion\textnormal{)}.
\end{fact}
\begin{proof} 
Given a filter $F$ in $\mathbb{A}$, we have 
\begin{eqnarray*}
\eta (F) &=& \bigcap\{\widehat{a}\mid a\in F\} \\
&=&\{F'\in UV(\mathbb{A})\mid \forall a\in F : F'\in\widehat{a}\} \\
&=&\{F'\in UV(\mathbb{A})\mid \forall a\in F : a\in F' \} \\
&=&\{F'\in UV(\mathbb{A})\mid F\subseteq F' \} \\
&=& \mathord{\Uparrow}F,
\end{eqnarray*}
where we recall that $\mathord{\Uparrow}F=\{G\in X\mid F\leqslant G\}$. By the same argument, any principal upset $\mathord{\Uparrow}F$ in the specialization order of the UV-space is equal to $\eta(F)$. Finally, it is clear that $F\subseteq F'$ iff $\eta(F)\supseteq \eta(F')$.
\end{proof}

\begin{fact}\label{IdealIso}  Let $\mathbb{A}$ be a BA and $X$ its dual UV-space. The map $\zeta$ is an isomorphism between the poset of ideals of $\mathbb{A}$ \textnormal{(}ordered by inclusion\textnormal{)} and $(\mathsf{O}\mathcal{RO}(X),\subseteq)$.
\end{fact}
\begin{proof} First, we show that for any ideal $I$ in $\mathbb{A}$, we have $\zeta(I)\in \mathsf{O}\mathcal{RO}(X)$. The set $\zeta(I)$ is a union of basic opens and hence is open. We claim that $\zeta(I)$ is also an $\mathcal{RO}$ set. To see that it is an $\leqslant$-upset, if $F\in \zeta(I)$, so for some $a\in I$, we have $F\in\widehat{a}$ and hence $a\in F$, then for any $F'\supseteq F$, we have $a\in F'$ and hence $F'\in\widehat{a}$, so $F'\in \zeta(I)$. Then to see that $\zeta(I)$ is an $\mathcal{RO}$ set, suppose $F\not\in \zeta(I)$, so for all $a\in I$, $F\not\in\widehat{a}$ and hence $a\not\in F$. Let $F'$ be the filter generated by $F\cup \{-a\mid a\in I\}$. We claim that $F'$ is proper. If not, then there are $b\in F$ and $a_1,\dots,a_n\in I$ such that $b\wedge -a_1\wedge\dots\wedge - a_n=\zero$, so $b\leq a_1\vee\dots\vee a_n$. Then since $F$ is a filter, $b\in F$ implies $a_1\vee\dots \vee a_n\in F$. But since $I$ is an ideal, $a_1,\dots,a_n\in I$ implies $a_1\vee\dots\vee a_n\in I$ and hence $a_1\vee\dots\vee a_n\not\in F$ by our choice of $F$. From this contradiction, we conclude that $F'$ is proper. Then since $-a\in F'$ for each $a\in I$, it follows that for any proper filter $F''\supseteq F'$, we have $F''\not\in \zeta(I)$. This shows that $\zeta(I)$ is an $\mathcal{RO}$ set by (\ref{ROeq}). This in turn completes the proof that $\zeta(I)\in \mathsf{O}\mathcal{RO}(X)$, and it is easy to see that $I\subseteq I'$ iff $\zeta(I)\subseteq \zeta(I')$.

Finally, to see that $f$ is surjective, given any $\mathsf{O}\mathcal{RO}$ subset $U$ of $UV(\mathbb{A})$, by the proof of Theorem \ref{MainRep} we have $U=\bigcup\{\widehat{a}\mid \widehat{a}\subseteq U\}$.  We claim that the set $I=\{a\mid \widehat{a}\subseteq U\}$ is an ideal in $\mathbb{A}$. If $a\in I$, so $\widehat{a}\subseteq U$, then for any $b\leq a$, we have $\widehat{b}\subseteq\widehat{a}$ and hence $\widehat{b}\subseteq U$, so $b\in I$. Finally, if $a,b\in I$, so $\widehat{a},\widehat{b}\subseteq U$ and hence $\widehat{a}\cup\widehat{b}\subseteq U$, then we have $\mathsf{int}_\leqslant(\mathsf{cl}_\leqslant(\widehat{a}\cup\widehat{b}))\subseteq \mathsf{int}_\leqslant(\mathsf{cl}_\leqslant(U))$. Since $U$ is an $\mathcal{RO}$ set, we have $\mathsf{int}_\leqslant(\mathsf{cl}_\leqslant(U))=U$, and then since $\mathsf{int}_\leqslant(\mathsf{cl}_\leqslant(\widehat{a}\cup\widehat{b}))=\widehat{a\vee b}$, it follows that $\widehat{a\vee b}\subseteq U$ and hence $a\vee b\in I$. Thus, $I$ is an ideal, and clearly $\zeta(I)=U$.
 \end{proof}
 
 \begin{fact}  Let $\mathbb{A}$ be a BA and $X$ its dual UV-space. The restriction of $\eta$ to principal filters is a dual isomorphism between the poset of principal filters  of $\mathbb{A}$ \textnormal{(}ordered by inclusion\textnormal{)} and $(\mathsf{CO}\mathcal{RO}(X),\subseteq)$.
 \end{fact}
 \begin{proof} The map $\mathord{\uparrow}a\mapsto \widehat{a}$ is the dual isomorphism, using the fact from Theorem \ref{MainRep} that $\widehat{a}\in\mathsf{CO}\mathcal{RO}(UV(\mathbb{A}))$. \end{proof}
 
 \begin{fact} Let $\mathbb{A}$ be a BA and $X$ its dual UV-space. The restriction of $\zeta$ to principal ideals is a dual isomorphism between the poset of principal ideals of $\mathbb{A}$ \textnormal{(}ordered by inclusion\textnormal{)} and $(\mathsf{CO}\mathcal{RO}(X),\subseteq)$.
 \end{fact}
 \begin{proof}  The map $\mathord{\downarrow}a\mapsto \widehat{a}$ is the dual isomorphism. \end{proof}
 
  \begin{fact} Let $\mathbb{A}$ be a BA and $X$ its dual UV-space. The restriction of $\eta$ to maximal filters is a bijection between the collection of maximal filters of $\mathbb{A}$ and the collection of singleton sets $\{x\}$ for $x\in\mathsf{Max}_\leqslant(X)$.
 \end{fact}
 
 \begin{proof} Since the specialization order $\leqslant$ of $X$ is the inclusion order $\subseteq$ on proper filters of $\mathbb{A}$, the elements of $\mathsf{Max}_\leqslant(X)$ are exactly the maximal filters of $\mathbb{A}$. By Fact \ref{EtaFact}, for any filter $F$, $\eta(F)=\mathord{\Uparrow}F$, so if $F$ is a maximal filter, then $\eta(F)=\mathord{\Uparrow}F=\{F\}$.\end{proof}
 
 \begin{fact} Let $\mathbb{A}$ be a BA and $X$ its dual UV-space. The restriction of $\zeta$ to maximal ideals is a bijection between the collection of maximal ideals of $\mathbb{A}$ and the collection of sets $X\setminus \mathord{\Downarrow}x$ for $x\in\mathsf{Max}_\leqslant(X)$.
\end{fact}
 
 \begin{proof} If $I$ is a maximal ideal in $\mathbb{A}$, then the complement $F$ of $I$ is a maximal filter in $\mathbb{A}$ and hence an element of $\mathsf{Max}_\leqslant(X)$. We claim that $\zeta(I)= \bigcup\{\widehat{a}\mid a\in I\}=X\setminus \mathord{\Downarrow}F$. For the $\subseteq$ inclusion, if $G\in \zeta(I)$, then for some $a\in I$, we have $G\in\widehat{a}$ and hence $a\in G$, which implies $G\not\subseteq F$. Conversely, if $G\in X\setminus \mathord{\Downarrow}F$, then $G\not\subseteq F$, so there is an $a\in G$ such that $a\not\in F$. Thus, we have an $a\in I$ such that $G\in \widehat{a}$ and hence $G\in \zeta(I)$.\end{proof}
 
 In Section \ref{CompletionSection} we will prove a correspondence between the \textit{normal} ideals of $\mathbb{A}$ and sets in $\mathsf{RO}(UV(\mathbb{A}))$.
 
 \subsection{Relativization} As one would expect by analogy with standard Stone duality, the operation on a UV-space dual to relativizing a BA to an element is the operation of taking a $\mathsf{CO}\mathcal{RO}$ subspace of a UV-space.

\begin{prop}\label{SubSpaceProp} Let $X$ be a UV-space. If $U\in\mathsf{CO}\mathcal{RO}(X)$, then $U$ with the subspace topology is a UV-space.
\end{prop}
\begin{proof} It is well known that every compact open subspace of a spectral space is again spectral.\footnote{This fact does not use any choice. To see that an open subspace of $X$ is sober, suppose $U$ is such a subspace. To prove that $U$ is sober, it suffices to show (see \cite[p.~2]{Picado2012}) that any open $V\subsetneq U$ is meet-irreducible iff it is the complement of the closure of a point. Let $V$ be an open proper subset of $U$, and suppose $V$ is meet-irreducible, so for all open $A,B\subseteq U$, if $A\cap B\subseteq V$, then $A\subseteq V$ or $B\subseteq V$. But then note that $V$ is also a meet-irreducible proper open subset of $X$. So by the sobriety of $X$, we have $V = X\setminus \mathsf{cl}\{x\}$ for some $x\in X$. Now if $x\in X\setminus U$ and hence $\mathsf{cl}\{x\}\subseteq X\setminus U$, then together $V\subsetneq U$ and $V = X\setminus \mathsf{cl}\{x\}$ imply $U = V$, a contradiction. Thus, $x\in U$ and $V = U \setminus \mathsf{cl}^U\{x\}$.}

Thus, since $X$ is a spectral space, so is the subspace induced by $U$. We denote the interior and closure operations given by the restriction of $\leqslant$ to $U$ by $\mathsf{int}_\leqslant^U$ and $\mathsf{cl}_\leqslant^U$, respectively.
It is easy to check that $\mathsf{CO}(U) = \{V\cap U \mid V\in \mathsf{CO}(X)\}$. We will now show that $\mathsf{CO}\mathcal{RO}(U) = \{V'\cap U\mid V'\in  \mathsf{CO}\mathcal{RO}(X)\}$. 
Let $V\subseteq U$. We first prove that 
\begin{equation}
V\in \mathsf{CO}\mathcal{RO}(U)\mbox{ iff }V\in \mathsf{CO}\mathcal{RO}(X).\label{Viff}
\end{equation}

Let $V\in \mathsf{CO}\mathcal{RO}(U)$. Then clearly $V\in\mathsf{CO}(X)$. We will show that $V\in\mathcal{RO}(X)$. Since $V$ is open in $U$, it is open in $X$. So $V$ is an $\leqslant$-upset, and $V\subseteq \mathsf{int}_\leqslant\mathsf{cl}_\leqslant (V)$. 
Now suppose  $x\in \mathsf{int}_\leqslant\mathsf{cl}_\leqslant (V)$. Then for each $y\in X$ with $x\leqslant y$, there is $z\in X$ with   $y\leqslant z$ and $z\in V$, which with $V\subseteq U$ implies $z\in U$. Thus, $x\in \mathsf{int}_\leqslant\mathsf{cl}_\leqslant (U)$, which implies $x\in U$ since $U\in\mathcal{RO}(X)$. Together $x\in\mathsf{int}_\leqslant\mathsf{cl}_\leqslant (V)$ and $x\in U$ imply $x\in \mathsf{int}^U_\leqslant\mathsf{cl}^U_\leqslant (V)$, which implies $x\in V$ since $V\in\mathcal{RO}(U)$. Thus, $\mathsf{int}_\leqslant\mathsf{cl}_\leqslant (V)\subseteq V$, so $V\in \mathcal{RO}(X)$. 

Conversely, suppose $V\in \mathsf{CO}\mathcal{RO}(X)$. Then clearly $V\in \mathsf{CO}(U)$. To show that $V\in \mathcal{RO}(U)$, suppose $x\in U$ but $x\not\in V$. Then since $V\in \mathcal{RO}(X)$, there is a $y\in X$ such that (a) $x\leqslant y$ and (b) for all $z\in X$ with $y\leqslant z$, we have $z\not\in V$. Since $U$ is an $\leqslant$-upset with $x\in U$, (a) implies $y\in U$. In addition, (b) implies that for all $z\in U$ with $y\leqslant z$, we have $z\not\in V$. Thus, we have shown that if $x\not\in V$, then there is a $y\in U$ such that $x\leqslant y$ and for all $z\in U$ with $y\leqslant z$, we have $z\not\in V$. Hence $V\in\mathcal{RO}(U)$.

The left-to-right direction of (\ref{Viff}) yields $\mathsf{CO}\mathcal{RO}(U)\subseteq \mathsf{CO}\mathcal{RO}(X)$. Now let $V'\in \mathsf{CO}\mathcal{RO}(X)$. Then $V'\cap U\in \mathsf{CO}\mathcal{RO}(X)$ and $V'\cap U\subseteq U$, so $V'\cap U\in \mathsf{CO}\mathcal{RO}(U)$ by the right-to-left direction of (\ref{Viff}). Therefore we have proved that $\mathsf{CO}\mathcal{RO}(U) = \{V'\cap U\mid  V'\in  \mathsf{CO}\mathcal{RO}(X)\}$.
 
 Next, we show that if $V\in \mathsf{CO}(U)$, then $\mathsf{int}^U_\leqslant (U\setminus V)\in  \mathsf{CO}(U)$. Note that for each $W\subseteq U$, we have $\mathsf{int}^U_\leqslant (W) = U\cap \mathsf{int}_\leqslant ((X\setminus U)\cup W)$. So  $\mathsf{int}^U_\leqslant (U\setminus V) = U\cap  \mathsf{int}_\leqslant ((X\setminus U)\cup (U\setminus V)) = U\cap  \mathsf{int}_\leqslant (X\setminus V)$. Since $X$ is a UV-space, $\mathsf{int}_\leqslant (X\setminus V)\in \mathsf{CO}\mathcal{RO}(X)$. Then as $U\in\mathsf{CO}(X)$ and  $\mathsf{CO}(X)$ is closed under finite intersections, $\mathsf{int}^U_\leqslant (U\setminus V)\in \mathsf{CO}(X)$. So $\mathsf{int}^U_\leqslant (U\setminus V)\in \mathsf{CO}(U)$.
 
 Finally, let $F$ be a filter in  $\mathsf{CO}\mathcal{RO}(U)$. Let $F'$ be the filter in $\mathsf{CO}\mathcal{RO}(X)$ generated by $F$. Then $F' = \mathsf{CO}\mathcal{RO}(x)$ for some $x\in X$. But then $x\in V$ for each $V\in F$. 
 So $x\in U$ and $F = \mathsf{CO}\mathcal{RO}(x)$. Thus, $U$ is a UV-space. \end{proof}

\begin{prop}\label{RelaProp} Let $X$ be a UV-space. For any $U\in\mathsf{CO}\mathcal{RO}(X)$, the relativization of the BA $\mathsf{CO}\mathcal{RO}(X)$ to $U$ is the dual of the subspace of $X$ induced by $U$.
\end{prop}

 \begin{proof} The proposition follows from two facts. First, by Proposition \ref{SubSpaceProp}, the subspace of $X$ induced by $U$ is a UV-space, so by Theorem \ref{SecondThm}.\ref{SecondThmB}, $U$ is homeomorphic to $UV(\mathsf{CO}\mathcal{RO}(U))$. Second, $\mathsf{CO}\mathcal{RO}(U) = \{V'\cap U\mid V'\in  \mathsf{CO}\mathcal{RO}(X)\}$ is the relativization of the BA $\mathsf{CO}\mathcal{RO}(X)$ to $U$.\end{proof}

 \subsection{Completeness} We now characterize the UV-duals of complete BAs.
 
 \begin{definition} A UV-space $X$ is \textit{complete} iff $\mathsf{int}( \mathsf{cl}(U))\in\mathsf{CO}\mathcal{RO}(X)$ for every open $U$.
 \end{definition}

 \begin{prop}\label{Comp} Let $\mathbb{A}$ be a BA and $X$ its dual UV-space. 
 \begin{enumerate}
 \item\label{Comp1} If $\{U_i\}_{i\in I}\subseteq \mathsf{CO}\mathcal{RO}(X)$, then $\{U_i\}_{i\in I}$ has a meet in $\mathsf{CO}\mathcal{RO}(X)$ iff 
 \[\mathsf{int}\underset{i\in I}{\bigcap}U_i\in\mathsf{CO}\mathcal{RO}(X),\]
 in which case \[\underset{i\in I}{\bigwedge}U_i=\mathsf{int}\underset{i\in I}{\bigcap}U_i.\]
 \item\label{Comp2} If $\{U_i\}_{i\in I}\subseteq \mathsf{CO}\mathcal{RO}(X)$, then $\{U_i\}_{i\in I}$ has a join in $\mathsf{CO}\mathcal{RO}(X)$ iff 
 \[\mathsf{int} (\mathsf{cl} \underset{i\in I}{\bigcup}U_i)\in\mathsf{CO}\mathcal{RO}(X),\]
 in which case \[\underset{i\in I}\bigvee U_i = \mathsf{int}( \mathsf{cl} \underset{i\in I}{\bigcup}U_i).\]
 \item\label{Comp3} $\mathbb{A}$ is complete iff $X$ is complete.
 \end{enumerate}
 \end{prop}
 \begin{proof} For part \ref{Comp1}, if $\mathsf{int}\underset{i\in I}{\bigcap}U_i\in\mathsf{CO}\mathcal{RO}(X)$, then clearly $\mathsf{int}\underset{i\in I}{\bigcap}U_i$ is the greatest lower bound in $\mathsf{CO}\mathcal{RO}(X)$ of $\{U_i\}_{i\in I}$. Conversely, if $\underset{i\in I}{\bigwedge}U_i$ exists in $\mathsf{CO}\mathcal{RO}(X)$, then we claim that $\underset{i\in I}{\bigwedge}U_i=\mathsf{int}\underset{i\in I}{\bigcap}U_i$. By the proof of Theorem \ref{MainRep}, for each $i\in I$, we have $U_i=\widehat{a_i}$ for some $a_i\in\mathbb{A}$, so $\underset{i\in I}{\bigwedge}U_i=\underset{i\in I}{\bigwedge}\widehat{a}_i$. Since $a\mapsto\widehat{a}$ is an isomorphism from $\mathbb{A}$ to $\mathsf{CO}\mathcal{RO}(X)$, we have $ \underset{i\in I}{\bigwedge}\widehat{a_i}= \widehat{\underset{i\in I}{\bigwedge}a_i}$. Thus, it suffices to show that $ \widehat{\underset{i\in I}{\bigwedge}a_i} = \mathsf{int} \underset{i\in I}{\bigcap} \widehat{a_i}$. Suppose $F\in  \widehat{\underset{i\in I}{\bigwedge}a_i}$. Then since $\widehat{\underset{i\in I}{\bigwedge}a_i}\subseteq \underset{i\in I}{\bigcap} \widehat{a_i}$ and $\widehat{\underset{i\in I}{\bigwedge}a_i}$ is open, we have $F\in\mathsf{int} \underset{i\in I}{\bigcap} \widehat{a_i}$. For the reverse inclusion, suppose $F\in \mathsf{int}\underset{i\in I}{\bigcap} \widehat{a_i}$, so there is a $U\in \mathsf{CO}\mathcal{RO}(X)$ such that $F\in U\subseteq \underset{i\in I}{\bigcap} \widehat{a_i}$. Then $U=\widehat{b}$ for some $b\in\mathbb{A}$, and $\widehat{b}\subseteq \underset{i\in I}{\bigcap} \widehat{a_i}$ implies that $b$ is a lower bound of $\{a_i\}_{i\in I}$ in $\mathbb{A}$, so $b\leq \underset{i\in I}{\bigwedge}a_i$. Then we have the following chain of implications: \[F\in \widehat{b} \Rightarrow b\in F \Rightarrow \underset{i\in I}{\bigwedge}a_i\in F \Rightarrow F\in \widehat{\underset{i\in I}{\bigwedge}a_i} .\] 
 
For part \ref{Comp2}, if $\underset{i\in I}{\bigvee}U_i$ exists in $\mathsf{CO}\mathcal{RO}(X)$, then we claim that $\underset{i\in I}{\bigvee}U_i=\mathsf{int}(\mathsf{cl}\underset{i\in I}{\bigcup}U_i)$. By the proof of Theorem \ref{MainRep}, for each $i\in I$, we have $U_i=\widehat{a_i}$ for some $a_i\in\mathbb{A}$, so $\underset{i\in I}{\bigvee}U_i=\underset{i\in I}{\bigvee}\widehat{a}_i$. Since $a\mapsto\widehat{a}$ is an isomorphism from $\mathbb{A}$ to $\mathsf{CO}\mathcal{RO}(X)$, we have $ \underset{i\in I}{\bigvee}\widehat{a_i}= \widehat{\underset{i\in I}{\bigvee}a_i}$. Thus, it suffices to show that $ \widehat{\underset{i\in I}{\bigvee}a_i} = \mathsf{int}(\mathsf{cl} \underset{i\in I}{\bigcup} \widehat{a_i})$. For the right-to-left inclusion, since $\underset{i\in I}{\bigcup} \widehat{a_i}\subseteq \widehat{\underset{i\in I}{\bigvee}a_i}$ and $\widehat{\underset{i\in I}{\bigvee}a_i}\in\mathsf{CO}\mathcal{RO}(X)=\mathsf{CRO}(X)$ (by Corollary \ref{Jakl}), we have $\mathsf{int}(\mathsf{cl}\underset{i\in I}{\bigcup} \widehat{a_i})\subseteq \mathsf{int}(\mathsf{cl}\widehat{\underset{i\in I}{\bigvee}a_i})=\widehat{\underset{i\in I}{\bigvee}a_i}$. For the left-to-right inclusion, since $\widehat{\underset{i\in I}{\bigvee}a_i}$ is open, it suffices to show $\widehat{\underset{i\in I}{\bigvee}a_i}\subseteq \mathsf{cl} \underset{i\in I}{\bigcup} \widehat{a_i}$. Consider any $F\in \widehat{\underset{i\in I}{\bigvee}a_i}$ and basic open neighborhood $U$ of $F$, so $U=\widehat{b}$ for some $b\in\mathbb{A}$. Then since  $F\in \widehat{b}$ and $F\in \widehat{\underset{i\in I}{\bigvee}a_i}$, we have $b\in F$ and $\underset{i\in I}{\bigvee}a_i\in F$, so $b\wedge \underset{i\in I}{\bigvee}a_i=\underset{i\in I}{\bigvee} (b\wedge a_i) \in F$.\footnote{Here we use the join-infinite distributive law for BAs, which says that if $\underset{i\in I}{\bigvee}a_i$ exists, then $\underset{i\in I}{\bigvee} (b\wedge a_i)$ exists and $b\wedge \underset{i\in I}{\bigvee}a_i=\underset{i\in I}{\bigvee} (b\wedge a_i)$ \cite[p.~47, Lem.~3]{Givant2009}.} Since $F$ is a proper filter, it follows that for some $i\in I$, $b\wedge a_i\neq 0$ and hence $\widehat{b}\cap\widehat{a_i}\neq\varnothing$. Thus, $\widehat{b}\cap  \underset{i\in I}{\bigcup} \widehat{a_i}\neq\varnothing$. This shows that $F\in \mathsf{cl} \underset{i\in I}{\bigcup} \widehat{a_i}$.

  For part \ref{Comp3}, suppose $X$ is complete. For any $\{a_i\}_{i\in I}\subseteq\mathbb{A}$, the set $\underset{i\in I}{\bigcup}\widehat{a_i}$ is open, so by the completeness of $X$, we have $\mathsf{int}( \mathsf{cl} \underset{i\in I}{\bigcup}\widehat{a_i})\in\mathsf{CO}\mathcal{RO}(X)$, in which case $\underset{i\in I}{\bigvee}a_i$ exists by part \ref{Comp2}. Conversely, suppose $\mathbb{A}$ is complete and $U$ is an open set in $X$. Then by Definition \ref{VOspace}.\ref{CloseProp}, we have that $U=\bigcup\{V\in\mathsf{CO}\mathcal{RO}(X)\mid V\subseteq U\}$. Since $\mathbb{A}$ is complete, so is the isomorphic $\mathsf{CO}\mathcal{RO}(X)$, so $\bigvee \{V\in\mathsf{CO}\mathcal{RO}(X)\mid V\subseteq U\}$ exists. Then by part \ref{Comp2}, $\bigvee \{V\in\mathsf{CO}\mathcal{RO}(X)\mid V\subseteq U\}=\mathsf{int}( \mathsf{cl}  \bigcup \{V\in\mathsf{CO}\mathcal{RO}(X)\mid V\subseteq U\})$, so $\mathsf{int}( \mathsf{cl} \bigcup \{V\in\mathsf{CO}\mathcal{RO}(X)\mid V\subseteq U\})\in\mathsf{CO}\mathcal{RO}(X)$, i.e., $\mathsf{int}( \mathsf{cl} U)\in\mathsf{CO}\mathcal{RO}(X)$. Hence $X$ is complete.
 \end{proof}
 
 \begin{remark} In contrast to the equality in Proposition \ref{Comp}.\ref{Comp2} for arbitrary joins, we observed in Proposition \ref{COROBA} that for finite joins, we have $U_1\vee\dots\vee U_n =\mathsf{int}_\leqslant (\mathsf{cl}_\leqslant(U_1\cup\dots\cup U_n))$. However, we cannot assert this equality for arbitrary joins, as it is refutable in ZF + Boolean Prime Ideal Theorem. To see this, suppose $F$ is a non-principal ultrafilter. Then $\bigwedge F=\zero$. For if $b$ is a lower bound of $F$, then since $F$ is non-principal, $b\not\in F$, and then since $F$ is an ultrafilter, $- b\in F$. But then $b\leq - b$, so $b=\zero$. Now since $\bigwedge F=\zero$, we have $\bigvee \{- a\mid a\in F\}=\one$, so $\bigvee \{- a\mid a\in F\}\in F$. Thus, we have $F\in \reallywidehat{\bigvee \{- a\mid a\in F\}}=\bigvee \{\widehat{- a}\mid a\in F\}$, yet clearly $F\not\in  \mathsf{int}_\leqslant (\mathsf{cl}_\leqslant \bigcup \{\widehat{- a}\mid a\in F\})$; since $F$ is an ultrafilter, it is maximal in $\leqslant$, so $F\in  \mathsf{int}_\leqslant (\mathsf{cl}_\leqslant \bigcup \{\widehat{- a}\mid a\in F\})$ implies $F\in \bigcup \{\widehat{- a}\mid a\in F\}$, contradicting the fact that $F$ is a proper filter. \end{remark}
 
 \begin{lemma}\label{SubspaceComplete} If $X$ is a complete UV-space and $U\in\mathsf{CO}\mathcal{RO}(X)$, then the subspace induced by $U$ is a complete UV-space.
 \end{lemma}
 
 \begin{proof} By Proposition~\ref{SubSpaceProp}, $U$ with the subspace topology is a UV-space. To show that $\mathsf{CO}\mathcal{RO}(U)$ is complete, it suffices to show that all meets exist. Thus, by Proposition~\ref{Comp}.\ref{Comp1}, it suffices to  show that for any $\{U_i\}_{i\in I}\subseteq \mathsf{CO}\mathcal{RO}(U)$,
we have $\mathsf{int}^U\bigcap_{i\in I} U_i\in \mathsf{CO}\mathcal{RO}(U)$. We show that $\mathsf{int}^U\bigcap_{i\in I} U_i = U\cap \mathsf{int}\bigcap_{i\in I} U_i $. Suppose $x\in \mathsf{int}^U \bigcap_{i\in I} U_i$. Then there is an open set $U_x\subseteq U$ such that $x\in U_x$ and
$U_x\subseteq U_i$ for each $i\in I$. But then $x\in U\cap \mathsf{int} \bigcap_{i\in I} U_i $. Conversely, if $x\in U\cap \mathsf{int}\bigcap_{i\in I} U_i$, then $x\in U$ and there is an open set $V_x$ such that $x\in V_x$ and $V_x\subseteq U_i$ for each $i\in I$. But then 
$V_x\subseteq U$ and so $x\in \mathsf{int}^U \bigcap_{i\in I} U_i$.

Therefore, by Proposition~\ref{Comp}.\ref{Comp1}, $\mathsf{int}^U(\bigcap_{i\in I} U_i)$ is the intersection of two 
$\mathsf{CO}\mathcal{RO}(U)$ sets and thus, by Proposition \ref{SubSpaceProp}, $\mathsf{int}^U \bigcap_{i\in I} U_i \in \mathsf{CO}\mathcal{RO}(U)$.
 \end{proof}
 
  \subsection{Atoms}
 
 Recall that an \textit{isolated point} of a space $X$ is an $x\in X$ such that $\{x\}$ is open. 
 
  \begin{prop}\label{Isolated} The map $a\mapsto \mathord{\uparrow}a$ is a bijection from the atoms of a BA to the isolated points of its dual UV-space.
 \end{prop}
 
 \begin{proof} If $a$ is an atom of the BA $\mathbb{A}$, then clearly $\widehat{a}=\{\mathord{\uparrow}a\}$, and $\widehat{a}$ is open in $UV(\mathbb{A})$, so $\mathord{\uparrow}a$ is an isolated point. If $a\neq b$, then $\mathord{\uparrow} a\neq\mathord{\uparrow}b$, so the map is injective. Finally, to see that the map is surjective, if $F$ is an isolated point, then $\{F\}$ is open and hence $\{F\}\in \mathsf{CO}\mathcal{RO}(UV(\mathbb{A}))$ by Definition \ref{VOspace}.\ref{CloseProp}. Thus, by the proof of Theorem \ref{MainRep}, there is some $a\in\mathbb{A}$ such that $\widehat{a}=\{F\}$, which implies that $a$ is an atom. For if $a$ is not an atom, then there is a $b< a$ with $b\neq \zero$, in which case the proper filters $\mathord{\uparrow}b$ and $\mathord{\uparrow}a$ are distinct and belong to $\widehat{a}$. Since $a$ is an atom, $\widehat{a}=\{\mathord{\uparrow}a\}$, so $F=\mathord{\uparrow}a$. \end{proof}

 \begin{cor} A BA is atomless iff the set of isolated points of its dual UV-space is empty.
 \end{cor}
 
 Let $X_{\mathrm{iso}}$ be the set of all isolated points of the space $X$ and $At(\mathbb{A})$ the set of all atoms of the BA $\mathbb{A}$.
 
 \begin{prop}\label{AtomicProp} Let $\mathbb{A}$ be a BA and $X$ its dual space. The following are equivalent:
 \begin{enumerate}
 \item $\mathbb{A}$ is atomic;
 \item $\mathsf{int}( \mathsf{cl} X_\mathrm{iso})=X$;
 \item the set of isolated points is dense in $X$, i.e., $\mathsf{cl} X_\mathrm{iso}=X$.
 \end{enumerate}
 \end{prop}
 
 \begin{proof}   1 $\Rightarrow$ 2. If $\mathbb{A}$ is atomic, then $\one = \bigvee \{a\in \mathbb{A}\mid a\in At(\mathbb{A})\}$. Then $X = \widehat{\one} = \reallywidehat{\bigvee \{a\in \mathbb{A}\mid a\in At(\mathbb{A})\}} = \bigvee \{\widehat{a}\in \mathbb{A}\mid a\in At(\mathbb{A})\}= \mathsf{int}(  \mathsf{cl} \bigcup \{\widehat{a}\mid a\in At(\mathbb{A})\})= \mathsf{int}(  \mathsf{cl} X_{\mathrm{iso}})$ by Propositions \ref{Comp}.\ref{Comp2} and \ref{Isolated}.
 
  2 $\Rightarrow$ 3. Since $\mathsf{int}( \mathsf{cl} X_\mathrm{iso})\subseteq \mathsf{cl} X_\mathrm{iso}$, $\mathsf{int}( \mathsf{cl} X_\mathrm{iso})=X$ implies $\mathsf{cl} X_\mathrm{iso}=X$.

 3 $\Rightarrow$ 1.  We need to show that $\one = \bigvee \{a\in A\mid a\in At(A)\}$. In dual terms this means that 
 $X = \widehat{\one} = \reallywidehat{\bigvee \{a\in A\mid a\in At(A)\}}$. By Propositions \ref{Comp}.\ref{Comp2} and \ref{Isolated}, we have $\reallywidehat{\bigvee \{a\in A\mid a\in At(A)\}} = \mathsf{int}(  \mathsf{cl} X_{\mathrm{iso}})$. As $\mathsf{cl}(X_{\mathrm{iso}}) = X$, we have $\mathsf{int}(\mathsf{cl}X_{\mathrm{iso}}) = \mathsf{int}X=X$.\end{proof}
 
  \subsection{Subalgebras and homomorphic images}
 
 We now characterize subalgebras and homomorphic images of BAs in terms of UV-spaces. 
 
 \begin{definition} Let $X$ and $Y$ be UV-spaces. An injective UV map $f:X\to Y$ is a \textit{UV-embedding} if for every $U\in\mathsf{CO}\mathcal{RO}(X)$ there is a $V\in \mathsf{CO}\mathcal{RO}(Y)$ such that $f[U]=f[X]\cap V$.
 \end{definition}
 
 \begin{fact}\label{thm: surj-inj} Let $\mathbb{A}$ and $\mathbb{B}$ be BAs and $h: \mathbb{A}\to \mathbb{B}$ a homomorphism. Let $h_+: UV(\mathbb{B})\to UV(\mathbb{A})$ be the UV-map dual to $h$. Then:
 \begin{enumerate}
 \item\label{surj1} if $h$ is injective, then $h_+$ is surjective;
 \item\label{inj1} if $h$ is surjective, then $h_+$ is a UV-embedding.
 \end{enumerate}
 \end{fact}
 \begin{proof} For part \ref{surj1}, consider a proper filter $F\in UV(\mathbb{A})$, and let $G = \{b\in \mathbb{A}\mid \exists a\in h[F]: a\leq b\}$. We show that $G$ is a proper filter such that $h^{-1}[G] = F$. Suppose $\zero_\mathbb{B}\in G$. Then $\zero_\mathbb{B}\in h[F]$, so there is an $a\in F$  such that $h(a) = \zero_\mathbb{B}$. As $F$ is proper, $a\neq \zero_\mathbb{A}$, which is a contradiction as $h$ is injective and $h(\zero_\mathbb{A}) = \zero_\mathbb{B}$. Now if $c, d\in G$, then there are $a,b\in F$ such that $h(a)\leq c$ and $h(b)\leq d$. Since $F$ is a filter, $a,b\in F$ implies $a\wedge b\in F$, so $h(a\wedge b)\in h[F]$.  Then since $h(a\wedge b) = h(a)\wedge h(b) \leq c\wedge d$, we have $c\wedge d\in G$. It is also obvious that $G$ is an upset. Thus, $G$ is a proper filter. 
 
 We now show that $h^{-1}[G] = F$. Clearly $F\subseteq h^{-1}[G]$. Suppose $a\in h^{-1}[G]$. Then $h(a)\in G$, so there is a $b\in F$ such that $h(b)\leq h(a)$. If $a\notin F$, then $b\nleq a$ and so $a\wedge b\neq b$. On the other hand, $h(a\wedge b) = h(a)\wedge h(b) = h(b)$, which is a contradiction as $h$ is injective. Therefore, $h^{-1}[G] = F$. As $h_+(G) = h^{-1}[G]$, we obtain that $h_+$ is a surjective {UV-map}.
 
 For part \ref{inj1}, let $F$ and $G$ be proper filters in $\mathbb{B}$ such that $F\neq G$. Then without loss of generality there is a $b\in F$ such that $b\notin G$. As $h$ is surjective there is an $a\in \mathbb{A}$ such that $h(a) = b$. Obviously, $a\in h^{-1}[F]$ and $a\notin h^{-1}[G]$. So $ h^{-1}[F]\neq h^{-1}[G]$, implying that $h_+$ is injective. Finally, we check the UV-embedding condition. Each $U\in \mathsf{CO}\mathcal{RO}(UV(\mathbb{B}))$ is of the form $\widehat{b}$ for some $b\in\mathbb{B}$. Since $h$ is surjective, there is an $a\in\mathbb{A}$ such that $h(a)=b$, so $h_+[\widehat{b}]=h_+[\widehat{h(a)}]$. Now it suffices to show \[h_+[\widehat{h(a)}]=h_+[UV(\mathbb{B})]\cap \widehat{a}.\] From left to right, suppose $F\in h_+[\widehat{h(a)}]$, so there is a $G\in \widehat{h(a)}$ such that $h_+(G)=F$. Since $G\in\widehat{h(a)}$, we have $h(a)\in G$. Since $h_+(G)=F$, we have $h^{-1}[G]=F$. From  $h(a)\in G$ and $h^{-1}[G]=F$, we have $a\in F$, so $F\in\widehat{a}$. From right to left, suppose $F\in h_+[UV(\mathbb{B})]\cap \widehat{a}$. Since $F\in h_+[UV(\mathbb{B})]$, there is a $G\in UV(\mathbb{B})$ such that $h_+(G)=F$ and hence $h^{-1}[G]=F$. Since $F\in \widehat{a}$, we have $a\in F$ and hence $h(a)\in G$. Thus,  $G\in\widehat{h(a)}$, which with $h_+(G)=F$ implies $F\in h_+[\widehat{h(a)}]$. This completes the proof.\end{proof}
 
 \begin{fact}\label{thm: surj-inj2}  Let $X$ and $Y$ be UV-spaces and $f: X\to Y$ a UV-map. Let $f^+: \mathsf{CO}\mathcal{RO}(Y)\to\mathsf{CO}\mathcal{RO}(X)$ be the homomorphism dual to $h$. Then:
 \begin{enumerate}
  \item\label{inj2} if $f$ is surjective, then $f^+$ is injective;
 \item\label{surj2} if $f$ is UV-embedding, then $f^+$ is surjective.
 \end{enumerate}
 \end{fact}
 
 \begin{proof} For part \ref{inj2}, suppose for $U,V\in\mathsf{CO}\mathcal{RO}(Y)$ that $U\neq V$. Suppose $y\in U\setminus V$. Since $f$ is surjective, there is an $x\in X$ such that $f(x)=y$, so $x\in f^{-1}[U]=f^+[U]$ but $x\not\in f^{-1}[V]=f^+(V)$. Hence $f^+$ is injective.
 
 For part \ref{surj2}, suppose $U\in\mathsf{CO}\mathcal{RO}(X)$. Then since $f$ is a UV-embedding, there is a $V\in\mathsf{CO}\mathcal{RO}(Y)$ such that $f[U]=f[X]\cap V$. Then $f^{-1}[f[U]]=f^{-1}[f[X]\cap V]=f^{-1}[f[X]]\cap f^{-1}[V]=X\cap f^{-1}[V]=f^{-1}[V]$. Since $f$ is injective, $f^{-1}[f[U]]=U$. Hence $U=f^{-1}[V]=f^+[V]$.\end{proof}

\begin{cor} $\,$
\begin{enumerate}
\item\label{subimage} There is a one-to-one correspondence between subalgebras of a BA $\mathbb{A}$ and images via onto UV-maps of its dual UV-space $X_\mathbb{A}$.
\item\label{homprinc} There is a one-to-one correspondence  between homomorphic images of a BA $\mathbb{A}$ and subspaces induced by principal upsets in the specialization order of the dual UV-space $X_\mathbb{A}$.
\end{enumerate}
\end{cor}
 \begin{proof} Part \ref{subimage} follows from Facts~\ref{thm: surj-inj}.\ref{surj1} and \ref{thm: surj-inj2}.\ref{inj2} and Theorem \ref{DualityThm}.
 
 Since there is a one-to-one correspondence between homomorphic images of $\mathbb{A}$ and filters of $\mathbb{A}$, part \ref{homprinc} follows directly from Fact~\ref{EtaFact}. However, we also sketch a more direct argument.   By Facts~\ref{thm: surj-inj}.\ref{inj1} and \ref{thm: surj-inj2}.\ref{surj2} and Theorem \ref{DualityThm}, there is a one-to-one correspondence between homomorphic images of $\mathbb{A}$ and UV-embeddings into its dual $X_\mathbb{A}$. Let $\mathbb{B}$ be a homomorphic image of $\mathbb{A}$ via $h$. Then $h_+:X_\mathbb{B}\to X_\mathbb{A}$ is a UV-embedding. First, since $h_+$ is an injective p-morphism, $h_+[X_\mathbb{B}]$ is a principal upset in the specialization order of $X_\mathbb{A}$. Second, if $Y$ is the subspace of $X_\mathbb{A}$ induced by $h_+[X_\mathbb{B}]$, we claim that $X_\mathbb{B}$ and $Y$ are homeomorphic via the bijection $h_+: X_\mathbb{B}\to Y$. Since $h_+$ is a continuous map from $X_\mathbb{B}$ to $X_\mathbb{A}$, it follows that $h_+$ is a continuous map from $X_\mathbb{B}$ to $Y$, and since $h_+$ is a UV-embedding from $X_\mathbb{B}$ to $X_\mathbb{A}$, it follows that $h_+$ is an open map from $X_\mathbb{B}$ to $Y$.  \end{proof}

\subsection{Products} The operation on UV-spaces dual to taking direct products of BAs is the following.

\begin{definition}\label{UVunionDef} The \textit{UV-sum} of disjoint UV-spaces $X$ and $Y$ is the space $X\bigcirc Y$ whose underlying set is $X\cup Y\cup (X\times Y)$ and whose topology is generated by the collection of sets
\[U\cup V \cup (U\times V)\]
for $U\in\mathsf{CO}\mathcal{RO}(X)$ and $V\in \mathsf{CO}\mathcal{RO}(Y)$.
\end{definition}

The following lemma is helpful for visualizing UV-sums.

\begin{lemma}\label{SumLem1} Given UV-spaces $X$ and $Y$ with specialization orders $\leqslant_X$ and $\leqslant_Y$, respectively, the specialization order in $X\bigcirc Y$ is given by:
\begin{eqnarray}
&& \leqslant_X\cup\leqslant_Y\cup \nonumber \\
&& \{\langle\langle x,y\rangle, x'\rangle\mid x\leqslant_X x'\}\cup \{\langle\langle x,y\rangle, y'\rangle\mid y\leqslant_Y y'\} \cup \nonumber \\
&& \{\langle\langle x,y\rangle,\langle x',y'\rangle\rangle\mid x\leqslant_X x',y\leqslant_Yy'\}.\label{SpecialDef}
\end{eqnarray}
\end{lemma}

\begin{proof} Suppose $\langle z,z'\rangle $ belongs to the set in (\ref{SpecialDef}), and $z$ belongs to an open set $U\cup V\cup (U\times V)$ of $X\bigcirc Y$. We must show that $z'$ also belongs to the set. There are five cases. If $z\leqslant_X z'$ (resp.~$z\leqslant_Y z'$), then $z\in U$ (resp.~$z\in V$), which with $U\in\mathsf{CO}\mathcal{RO}(X)$ (resp.~$V\in \mathsf{CO}\mathcal{RO}(Y)$) implies $z'\in U$ (resp.~$z'\in V$) and hence $z'\in U\cup V\cup (U\times V)$. On the other hand, if $z=\langle x,y\rangle $, then $\langle x,y\rangle\in U\times V$, so $x\in U$ and $y\in V$. Therefore, if $x\leqslant_X z'$ (resp.~$y\leqslant_Y z'$), then $z'\in U$ (resp.~$z'\in V$) and hence $z'\in U\cup V\cup (U\times V)$. Similarly, if $z'=\langle x',y'\rangle$, and $x\leqslant_X x'$ and $y\leqslant_Y y'$, then $x'\in U$ and $y'\in V$, so $z'\in U\times V$. This completes the proof that $z\leqslant_{X\bigcirc Y}z'$.

Conversely, suppose $\langle z,z'\rangle$ does not belong to the set in (\ref{SpecialDef}). Again there are five cases. For example, if $z,z'\in X$, it follows that $z\not\leqslant_X z'$, so there is an open set $U$ of $X$ such that $z\in U$ but $z'\not\in U$, and $U$ is open in $X\bigcirc Y$, so $z\not\leqslant_{X\bigcirc Y}z'$. Similarly, if $z=\langle x,y\rangle$ and $z'\in X$, it follows that $x\not\leqslant_X z'$, so there is an open set $U$ of $X$ such that $x\in U$ but $z'\not\in U$. Thus, $\langle x,y\rangle \in U\cap Y\cup ( U\times Y)$ but $z'\not\in U\cup Y\cup (U\times Y)$, so $\langle x,y\rangle\not\leqslant_{X\bigcirc Y}z'$. The other cases are analogous.
\end{proof}

\begin{example} For finite UV-spaces, the UV-sum is easily drawn. Figure \ref{Figure} shows the UV-sum of the UV-duals of the four-element and two-element BAs, \textbf{4} and \textbf{2} (recall Corollary \ref{UVspectral}.\ref{UVspectral4.5}). Solid lines indicate the specialization order $\leqslant$ in $UV(\mathbf{4})$, so $x\leqslant y_1$ and $x\leqslant y_2$. Dashed lines indicate the new part of the relation defined in Lemma \ref{SumLem1}. Note that $UV(\mathbf{4})\bigcirc UV(\mathbf{2})=UV(\mathbf{4}\times\mathbf{2})$ in line with Proposition \ref{UnionProduct}.
\end{example}

\begin{figure}[h]
\begin{center}
\begin{tikzpicture}[->,>=stealth',shorten >=1pt,shorten <=1pt, auto, node distance=2in,thick,every loop/.style={<-,shorten <=1pt}] \tikzstyle{every state}=[fill=gray!20,draw=none,text=black]

\node  at (-6,.5) {{$UV(\mathbf{4})$}};
\node (x) at (-6,1.25) {{$x$}};
\node (y1) at (-7,2.5) {{$y_1$}};
\node (y2) at (-5,2.5) {{$y_2$}};

\path (x) edge[-] node {{}} (y1); 
\path (x) edge[-] node {{}} (y2); 

\node  at (-3.5,.5) {{$UV(\mathbf{2})$}};
\node (z) at (-3.5,2.5) {{$z$}};

\node  at (0,-.75) {{$UV(\mathbf{4})\bigcirc UV(\mathbf{2})$}};
\node  at (0,-1.25) {{$UV(\mathbf{4}\times\mathbf{2})$}};

\node  (bot) at (0,0) {{$\langle x,z\rangle$}}; 
\node (A) at (-2,1.25) {{$x$}}; 
\node (B) at (0,1.25) {{$\langle y_1,z\rangle$}}; 
\node (C) at (2,1.25) {{$\langle y_2,z\rangle$}};

 \node (A') at (-2,2.5) {{$y_1$}}; 
 \node (B') at (0,2.5) {{$y_2$}}; 
\node (C') at (2,2.5) {{$z$}};

\path (A) edge[-] node {{}} (A'); 
\path (A) edge[-] node {{}} (B'); 
\path (B) edge[-,dashed] node {{}} (A'); 
\path (B) edge[-,dashed] node {{}} (C'); 
\path (C) edge[-,dashed] node {{}} (C'); 
\path (C) edge[-,dashed] node {{}} (B'); 
\path (A) edge[-,dashed] node {{}} (bot); 
\path (B) edge[-,dashed] node {{}} (bot); 
\path (C) edge[-,dashed] node {{}} (bot);

\end{tikzpicture}
\end{center}
\caption{UV-sum of the UV-duals of the BAs \textbf{4} and \textbf{2}.}\label{Figure}
\end{figure}
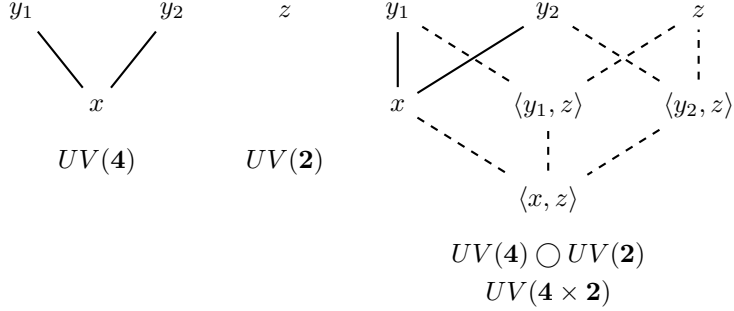

\begin{prop}\label{UnionProduct} For any BAs $\mathbb{A}$ and $\mathbb{B}$, $UV(\mathbb{A})\bigcirc UV(\mathbb{B})$ is homeomorphic to $UV(\mathbb{A}\times\mathbb{B})$.
\end{prop}

\begin{proof} Given $F\in  UV(\mathbb{A}\times\mathbb{B})$, so that $F$ is a proper filter in $\mathbb{A}\times \mathbb{B}$, we have that $F_\mathbb{A}=\{a\mid \exists b :\langle a,b\rangle\in F \}$ and $F_\mathbb{B}=\{b\mid \exists a :\langle a,b\rangle\in F \}$ are filters in $\mathbb{A}$ and $\mathbb{B}$, respectively, at least one of which is a proper filter. We define a function $h$ from $UV(\mathbb{A}\times\mathbb{B})$ to $UV(\mathbb{A})\bigcirc UV(\mathbb{B})$ as follows:
\[h(F)=\begin{cases}F_\mathbb{A} &\mbox{if }F_\mathbb{B}\mbox{ is improper} \\
F_\mathbb{B} &\mbox{if }F_\mathbb{A}\mbox{ is improper} \\
\langle F_\mathbb{A},F_\mathbb{B}\rangle & \mbox{otherwise}.
 \end{cases}\]
 We claim that $h$ is a homeomorphism from $UV(\mathbb{A}\times\mathbb{B})$ to $UV(\mathbb{A})\bigcirc UV(\mathbb{B})$. Clearly $h$ is injective. For surjectivity, suppose $G\in UV(\mathbb{A})\bigcirc UV(\mathbb{B})$. If $G\in UV(\mathbb{A})$, then for $F=\{\langle a,b\rangle\mid a\in G,b\in\mathbb{B}\}\in UV(\mathbb{A}\times\mathbb{B})$, we have that $G=F_\mathbb{A}$ and $F_\mathbb{B}$ is improper,  so $G=h(F)$. Similarly, if $G\in UV(\mathbb{B})$, then for $F=\{\langle a,b\rangle\mid a\in \mathbb{A},b\in G\}\in UV(\mathbb{A}\times\mathbb{B})$, we have that $G=F_\mathbb{B}$ and $F_\mathbb{A}$ is improper, so $G=h(F)$. Finally, if $G=\langle G^\mathbb{A},G^\mathbb{B}\rangle$ for $G^\mathbb{A}\in UV(\mathbb{A})$ and $G^\mathbb{B}\in UV(\mathbb{B})$, then $G^\mathbb{A}\times G^\mathbb{B}\in UV(\mathbb{A}\times\mathbb{B})$ and  $G=h(G^\mathbb{A}\times G^\mathbb{B})$, since $(G^\mathbb{A}\times G^\mathbb{B})_\mathbb{A}=G^\mathbb{A}$ and $(G^\mathbb{A}\times G^\mathbb{B})_\mathbb{B}=G^\mathbb{B}$. Thus, $h$ is surjective.
 
 To show that $h$ is continuous, it suffices to show that the inverse image of each basic open is open. By Definition \ref{UVunionDef}, each basic open in $UV(\mathbb{A})\bigcirc UV(\mathbb{B})$ is of the form $U\cup V\cup (U\times V)$ for $U\in \mathsf{CO}\mathcal{RO}(UV(\mathbb{A}))$ and $V\in \mathsf{CO}\mathcal{RO}(UV(\mathbb{B}))$. By the proof of Theorem \ref{MainRep}, $U=\widehat{a}$ and $V=\widehat{b}$ for some $a\in\mathbb{A}$ and $b\in\mathbb{B}$, so our basic open in $UV(\mathbb{A})\bigcirc UV(\mathbb{B})$ is $\widehat{a}\cup\widehat{b}\cup (\widehat{a}\times\widehat{b})$. Then we have:
 \begin{eqnarray*}
 h^{-1}[\widehat{a}\cup\widehat{b}\cup (\widehat{a}\times\widehat{b})] &=&  h^{-1}[\widehat{a}]\cup  h^{-1}[\widehat{b}]\cup  h^{-1}[\widehat{a}\times\widehat{b}] \\
 &=& \widehat{\langle a,\zero\rangle} \cup \widehat{\langle \zero,b\rangle} \cup \widehat{\langle a,b\rangle},
 \end{eqnarray*}
 so $h^{-1}[\widehat{a}\cup\widehat{b}\cup (\widehat{a}\times\widehat{b})] $ is a union of basic opens in $UV(\mathbb{A}\times\mathbb{B})$.
 
 Finally, to see that $h^{-1}$ is continuous, for any basic open set $\widehat{\langle a,b\rangle}$ of $UV(\mathbb{A}\times\mathbb{B})$, we have:
 \begin{eqnarray*}
 \widehat{\langle a,b\rangle} &=& \{F\in\mathrm{PropFilt}(\mathbb{A}\times\mathbb{B})\mid \langle a,b\rangle\in F\mbox{ and } F_\mathbb{B}\mbox{ improper} \} \cup \\
  && \{F\in\mathrm{PropFilt}(\mathbb{A}\times\mathbb{B})\mid \langle a,b\rangle\in F \mbox{ and }F_\mathbb{A}\mbox{ improper} \} \cup \\
  && \{F\in\mathrm{PropFilt}(\mathbb{A}\times\mathbb{B})\mid \langle a,b\rangle\in F\mbox{ and } F_\mathbb{A}, F_\mathbb{B}\mbox{ proper} \},
 \end{eqnarray*}
which implies
 \begin{eqnarray*}
 h[\widehat{\langle a,b\rangle}]&=& \widehat{a}\cup \widehat{b}\cup (\widehat{a}\times\widehat{b}),
 \end{eqnarray*}
 so that $h[\widehat{\langle a,b\rangle}]$ is basic open in $UV(\mathbb{A})\bigcirc UV(\mathbb{B})$.
\end{proof}

\begin{cor} For any UV-spaces $X$ and $Y$, $X\bigcirc Y$ is a UV-space.
\end{cor}
\begin{proof} By Theorem \ref{SecondThm}.\ref{SecondThmB}, $X$ and $Y$ are respectively homeomorphic to $UV(\mathsf{CO}\mathcal{RO}(X))$ and $UV(\mathsf{CO}\mathcal{RO}(Y))$, which implies that $X\bigcirc Y$ is homeomorphic to $UV(\mathsf{CO}\mathcal{RO}(X))\bigcirc UV(\mathsf{CO}\mathcal{RO}(Y))$. By Proposition \ref{UnionProduct}, $UV(\mathsf{CO}\mathcal{RO}(X)) \bigcirc UV(\mathsf{CO}\mathcal{RO}(Y))$ is homeomorphic to $UV(\mathsf{CO}\mathcal{RO}(X)\times\mathsf{CO}\mathcal{RO}(Y))$, which is a UV-space by Theorem \ref{SecondThm}.\ref{SecondThmA}. Thus, by the two homeomorphisms, $X\bigcirc Y$ is a UV-space.
\end{proof}

\begin{cor} For any UV-spaces $X$ and $Y$, $\mathsf{CO}\mathcal{RO}(X\bigcirc Y)$ is isomorphic to $\mathsf{CO}\mathcal{RO}(X)\times \mathsf{CO}\mathcal{RO}(Y)$.
\end{cor}

\begin{proof} Apply Proposition \ref{UnionProduct} and duality (Theorem \ref{DualityThm}). 
\end{proof}

\begin{remark}
Another natural question is how one can characterize products in the category of UV-spaces with UV-maps, which will be the duals of coproducts in the category of BAs with BA homomorphisms. We cannot characterize the product of UV-spaces $X$ and $Y$ as a topological space based on the Cartesian product of the underlying sets of $X$ and $Y$. E.g., if we take the Cartesian product of two copies of the three-element set underlying $UV(\mathbf{4})$ (Figure \ref{Figure}),  then we obtain a set with nine elements; this cannot be the underlying set of any poset obtained from a BA by deleting its top element, so by Corollary \ref{UVspectral}.\ref{UVspectral4.5} it cannot be the underlying set of a UV-space. We leave for future work the problem of characterizing products in the category of UV-spaces, which is reminiscent of the open problem of characterizing products in the category of Esakia spaces \cite{Esakia2018}.\end{remark}

\subsection{Completions}\label{CompletionSection} The \textit{canonical extension} of a BA $\mathbb{A}$, as defined in \cite{Gehrke2001}, is the unique (up to isomorphism) complete BA $\mathbb{B}$ for which there is a Boolean embedding $e$ of $\mathbb{A}$ into $\mathbb{B}$ such that every element of $\mathbb{B}$ is a join of meets of $e$-images of elements of $\mathbb{A}$, and for any sets $S$ and $T$ of elements of $\mathbb{A}$, if $\bigwedge e[S]\leq \bigvee e[T]$, then there are finite sets $S'\subseteq S$ and $T'\subseteq T$ such that $\bigwedge S'\leq \bigvee T'$. It is shown in \cite[\S~5.6]{Holliday2018} that the canonical extension of a BA $\mathbb{A}$ can be constructed without choice as the BA of all regular open upsets in the poset of proper filters of $\mathbb{A}$ ordered by inclusion. Putting this in terms of UV-spaces, we have the following.

\begin{theorem}\label{CanonicalExt} Let $\mathbb{A}$ be a BA and $X$ its dual UV-space. Then $\mathcal{RO}(X)$ is \textnormal{(}up to isomorphism\textnormal{)} the canonical extension of $\mathbb{A}$.
\end{theorem}

 The \textit{MacNeille completion} of a BA $\mathbb{A}$ is the unique (up to isomorphism) complete BA $\mathbb{B}$ for which there is a Boolean embedding $e$ of $\mathbb{A}$ into $\mathbb{B}$ such that every non-minimum element of $\mathbb{B}$ is above the $e$-image of some non-minimum element of $\mathbb{A}$ (see, e.g., \cite[Ch.~25]{Givant2009}). The MacNeille completion of $\mathbb{B}$ may be constructed as the lattice of normal ideals of $\mathbb{B}$ ordered by inclusion; an ideal $I$ of $\mathbb{B}$ is normal iff $I=I^{u\ell}$, where for any $A\subseteq \mathbb{B}$, $A^u$ is the set of upper bounds of $A$, and $A^\ell$ is the set of lower bounds of $A$.\footnote{The MacNeille completion of a BA $\mathbb{A}$ can also be constructed as the BA of all regular open upsets in the poset that results from deleting the bottom element of $\mathbb{A}$ and reversing the restricted order \cite[\S~5.6]{Holliday2018}.} 

\begin{theorem} Let $\mathbb{A}$ be a BA and $X$ its dual space. Then $\mathsf{RO}(X)$ is \textnormal{(}up to isomorphism\textnormal{)} the MacNeille completion of $\mathbb{A}$.\end{theorem}

\begin{proof} We show an order isomorphism between $\mathsf{RO}(X)$ and the set of normal ideals of $\mathbb{A}$ ordered by inclusion. It suffices to define an inclusion-preserving map $r$ from normal ideals to $\mathsf{RO}(X)$ and an inclusion-preserving map $i$ from $\mathsf{RO}(X)$ to normal ideals such that $i(r(I))=I$ and $r(i(U))=U$.

Suppose $I$ is a normal ideal, so $I=I^{u\ell}$. Let $r(I):=\bigcup \{\widehat{c}\mid c\in I\}$. To see that $r(I)\in\mathsf{RO}(X)$, let
$U:=\bigcup \{\widehat{-a}\mid a\in I^u\}$. Then as in the proof of Proposition 4.3, we have
\begin{eqnarray*}
U^*&=&\bigcup\{\widehat{c}\mid \forall a\in I^u\;\, \mathord{-}a\wedge c=0\}\\
&=&\bigcup\{\widehat{c}\mid \forall a\in I^u\;\,  c\leq a\}\\
&=&\bigcup \{\widehat{c}\mid c\in I^{u\ell}\}\\
&=&\bigcup \{\widehat{c}\mid c\in I\} = r(I).
\end{eqnarray*}
Thus, $r(I)\in\mathsf{RO}(X)$. Clearly $I\subseteq J$ implies $r(I)\subseteq r(J)$.

In the other direction, suppose $V\in \mathsf{RO}(X)$. Let $i(V)=\{-b\mid \widehat{b}\subseteq V^*\}^\ell$. It is easy to see that for any $S\subseteq\mathbb{A}$,  $S^\ell$ is a normal ideal, so $i(V)$ is a normal ideal. Also observe that $i$ is inclusion-preserving:
\begin{eqnarray*}
&& V\subseteq U \\
&\Rightarrow& U^*\subseteq V^*\\
&\Rightarrow &\{-b\mid \widehat{b}\subseteq U^*\}\subseteq \{-b\mid \widehat{b}\subseteq V^*\} \\
&\Rightarrow & \{-b\mid \widehat{b}\subseteq V^*\}^\ell \subseteq \{-b\mid \widehat{b}\subseteq U^*\}^\ell \\
&\Rightarrow & i(V)\subseteq i(U).
\end{eqnarray*}

Next, observe:
\begin{eqnarray*}
i(r(I)) &=&i\big(\bigcup \{\widehat{c}\mid c\in I\}\big) \\
&=& \{-b\mid \widehat{b}\subseteq\big(\bigcup \{\widehat{c}\mid c\in I\}\big)^*\}^\ell \\
&=& \{-b\mid \widehat{b}\subseteq\bigcup \{\widehat{d}\mid \forall c\in I\;\, c\wedge d=0\}\}^\ell \\
&=& \{-b\mid \widehat{b}\subseteq\bigcup \{\widehat{d}\mid \forall c\in I\;\, c\leq -d\}\}^\ell \\
&=& \{-b\mid \mathord{\uparrow}b\in \bigcup\{\widehat{d}\mid \forall c\in I\;\, c\leq -d\}^\ell  \\
&=& \{-b\mid \exists d: b\leq d\mbox{ and } \forall c\in I\;\, c\leq -d\}^\ell \\
&=& \{-b\mid \forall c\in I\;\, c\leq -b\}^\ell \\
&=& I^{u\ell}=I.
\end{eqnarray*}

Finally, observe:
\begin{eqnarray*}
r(i(U)) &=&r(\{-b\mid \widehat{b}\subseteq U^*\}^\ell)\\
&=&\bigcup \{\widehat{c}\mid c\in \{-b\mid \widehat{b}\subseteq U^*\}^\ell\}. \\
&=&\bigcup \{\widehat{c}\mid \forall b\,(\widehat{b}\subseteq U^*\Rightarrow c\leq -b)\}\\
&=&\bigcup \{\widehat{c}\mid \forall b\,(\widehat{b}\subseteq U^*\Rightarrow b\wedge c=0)\} \\
&=&U^{**} =U.
\end{eqnarray*}
This completes the proof.\end{proof}

\section{Example applications}\label{ApplicationSection} In this section, we apply our duality to prove some basic theorems about BAs in Propositions \ref{ChainsAnti}, \ref{CompleteBAProp}, and \ref{SubalgebrasProp}. 

\subsection{Chains and antichains in BAs} By an \textit{antichain} in a BA, we mean a collection $C$ of elements such that for all $x,y\in C$ with $x\neq y$, we have $x\wedge y=\zero$.

\begin{prop}\label{ChainsAnti} Every infinite BA contains infinite chains and infinite antichains.
\end{prop}

\begin{proof} By duality, it suffices to show that in any infinite UV-space $X$, there is an infinite descending chain $U_0\supsetneq U_1\supsetneq\dots$ of sets from  $\mathsf{CO}\mathcal{RO}(X)$, as well as an infinite family of pairwise disjoint sets from $\mathsf{CO}\mathcal{RO}(X)$. For this it suffices to show that for every infinite $U\in \mathsf{CO}\mathcal{RO}(X)$ (note that $X$ is such a $U$), there is an infinite $U'\in \mathsf{CO}\mathcal{RO}(X)$ with $U\supsetneq U'$ and $U\cap\neg U'\neq\varnothing$.  For then by DC, there is an infinite descending chain $U_0\supsetneq U_1\supsetneq \dots$ of  sets from $\mathsf{CO}\mathcal{RO}(X)$ with $U_i\cap \neg U_{i+1}\neq\varnothing$ for each $i\in\mathbb{N}$, in which case $\{U_0\cap \neg U_1,U_1\cap \neg U_2,\dots\}$ is our antichain.

Assume $U\in \mathsf{CO}\mathcal{RO}(X)$ is infinite. Since $X$ is $T_0$, there are $x,y\in U$ such that $x\not\leqslant y$. Then by the separation property of UV-spaces, there is a $V\in\mathsf{CO}\mathcal{RO}(X)$ such that $x\in V$ and $y\not\in V$, which with $y\in U$ and $U,V\in\mathcal{RO}(X)$ implies that there is a $z\geqslant y$ such that $z\in U\cap \neg V$. Since $U,V\in \mathsf{CO}\mathcal{RO}(X)$, we have $U\cap V,U\cap \neg V\in \mathsf{CO}\mathcal{RO}(X)$ by Definition \ref{VOspace}.\ref{CloseProp}; and since  $z\in U\cap \neg V$ and $x\in U\cap V$, we have $z\in U\cap \neg (U\cap V)\neq \varnothing$ and $x\in U\cap \neg (U\cap \neg V)\neq \varnothing$. Thus, if $U\cap V$ is infinite, then we can set $U':=U\cap V$, and otherwise we claim that $U\cap\neg V$ is infinite, in which case we can set $U':=U\cap\neg V$. Since $U\in \mathcal{RO}(X)$, we may regard $U$ as a separative partial order. Given $V\in\mathcal{RO}(X)$, we have  $U\cap V,U\cap \neg V\in \mathcal{RO}(U)$ and $U\cap \neg V=\neg_U(U\cap V)$, where $\neg_U$ is the complement operation in $\mathcal{RO}(U)$. Then since $U$ is infinite, by Lemma \ref{EitherInfinite} either $U\cap V$ or $\neg_U(U\cap V)$ is infinite, as desired.\end{proof}

\subsection{Products of BAs} Before our second example application in Proposition \ref{CompleteBAProp}, we prove a preliminary lemma. Recall from Proposition \ref{SubSpaceProp} that a subspace of a UV-space induced by a $\mathsf{CO}\mathcal{RO}$ set is also a UV-space.

\begin{lemma}\label{SumLem} If $X$ is a UV-space and $U\in\mathsf{CO}\mathcal{RO}(X)$, then $X$ is homeomorphic to the UV-sum of the subspaces induced by $U$ and $\neg U$, respectively.
\end{lemma}
\begin{proof} By Corollary \ref{UVspectral}.\ref{UVspectral4}, $(X,\leqslant)$ has a meet $x\sqcap y$ for any two elements $x,y\in X$. We define $f\colon U\bigcirc \neg U\to X$ as follows: if $z\in U\cup \neg U$, then $f(z)=z$; otherwise $z=\langle x,y\rangle$ for $x\in U$ and $y\in\neg U$, so we define $f(\langle x,y\rangle)=x\sqcap y$. That $f$ is a bijection follows from Corollary \ref{UVspectral}.\ref{UVspectral5}. To see that $f$ is continuous, we show that the inverse image of each basic open is open. Given $V\in\mathsf{CO}\mathcal{RO}(X)$, we have:
\begin{eqnarray*}
f^{-1}[V] &=& (U\cap V)\cup (\neg U\cap V)\cup \{\langle x,y\rangle\mid x\in U,y\in \neg U,x\sqcap y\in V\} \\
&=& (U\cap V)\cup (\neg U\cap V)\cup ((U\cap V) \times (\neg U\cap V)),
\end{eqnarray*}
where we have used the fact that $V$ is a filter with respect to $\sqcap$ (Corollary \ref{UVspectral}.\ref{UVspectral4}). Since $U\cap V\in \mathsf{CO}\mathcal{RO}(U)$ and $\neg U\cap V\in\mathsf{CO}\mathcal{RO}(\neg U)$, it follows from the above equation and Definition \ref{UVunionDef} that $f^{-1}[V]$ is open in $U\bigcirc\neg U$.

Finally, to see that $f^{-1}$ is continuous, each basic open of $U\bigcirc\neg U$ is of the form $V\cup V'\cup (V\times V')$ for $V\in\mathsf{CO}\mathcal{RO}(U)$ and $V'\in\mathsf{CO}\mathcal{RO}(\neg U)$ by Definition \ref{UVunionDef}. Then $V,V'\in\mathsf{CO}\mathcal{RO}(X)$ and 
\begin{eqnarray*}
f[V\cup V'\cup (V\times V')] &=& f[V]\cup f[V']\cup f[V\times V'] \\
&=& V\cup V' \cup \{x\sqcap y\mid x\in V,\, y\in V'\}\\
&=& V\vee V' \in\mathsf{CO}\mathcal{RO}(X),
\end{eqnarray*}
where the last equality uses Corollary \ref{UVspectral}.\ref{UVspectral6}.
\end{proof}

\begin{prop}\label{CompleteBAProp} Any complete BA is isomorphic to the product of a complete and atomless BA and a complete and atomic BA.
\end{prop}

\begin{proof} By duality, it suffices to show that any complete UV-space $X$ is the UV-sum of a complete UV-space with no isolated points and a complete UV-space in which the isolated points form a dense subset. Since $X$ is complete, $U:=\mathsf{int} (\mathsf{cl} X_\mathrm{iso})\in\mathsf{CO}\mathcal{RO}(X)$. Form the subspaces induced by $U$ and $\neg U$. By Lemma \ref{SumLem}, $X$ is homeomorphic to the UV-sum of these subspaces. By Lemma \ref{SubspaceComplete}, both subspaces are complete UV-spaces. Clearly $(\neg U)_\mathrm{iso}\subseteq X_\mathrm{iso}$, and $X_\mathrm{iso}=U_\mathrm{iso}$, which with $\neg U\cap U=\varnothing$ implies $(\neg U)_\mathrm{iso}=\varnothing$. Thus, the subspace induced by $\neg U$ has no isolated points. Finally, in the subspace induced by $U$, we have 
\[\mathsf{int}^U\mathsf{cl}^U U_\mathrm{iso}=U\cap\mathsf{int}(\mathsf{cl} U_\mathrm{iso})=U\cap U=U,\] which implies $\mathsf{cl}^U U_\mathrm{iso}=U$ by Proposition \ref{AtomicProp}.\end{proof}

\subsection{Subalgebras of BAs} Let $\mathbb{B}_n$ be the finite Boolean algebra with $n$ atoms. As our final example, we will prove using our duality that every infinite BA contains subalgebras isomorphic to $\mathbb{B}_n$ for each positive integer $n$. First, we prove some preliminary results about UV-spaces.

\begin{definition} Let $X$ be a UV-space and $\{U_0,\dots,U_n\}$ a family of $\mathsf{CO}\mathcal{RO}(X)$ sets. We say that $\{U_0,\dots,U_n\}$ is a \textit{regular partition} of $X$ iff $U_0,\dots,U_n$ are pairwise disjoint and $X=U_0\vee\dots \vee U_n$.
\end{definition}

\begin{prop}\label{prop: partition}
Let $X$ be an infinite UV-space. For each $n\in \omega$, there is a family $\{V_0,\dots, V_n\}$ of $\mathsf{CO}\mathcal{RO}$ sets that is a regular partition of $X$.
\end{prop}
\begin{proof} Consider the antichain $\{U_0\cap \neg U_1$, $U_1\cap\neg U_2$, \dots , $U_{n-1}\cap\neg U_n\}$ constructed in the proof of Proposition \ref{ChainsAnti}. Let $U_{n+1}=\varnothing$. We claim that the antichain $\{U_0\cap \neg U_1,\dots, U_{n-1}\cap\neg U_n, U_n\cap \neg U_{n+1}\}$ is a regular partition, i.e., its join is $X$. Using the equation for join in terms of $\mathsf{int}_\leqslant \mathsf{cl}_\leqslant$ and union (Proposition \ref{COROBA}), it suffices to show that for every $x\in X$, there is a $y\geqslant x$ such that $y\in U_i\cap U_{i+1}$ for some $i\in\{0,\dots, n\}$. If $x\in \neg U_1$, then since $U_0=X$, we have $x\in U_0\cap\neg U_1$, so we take $y=x$ and $i=0$. If $x\not\in U_1$, then there is an $x_1\geqslant x$ such that $x_1\in U_1$. Now if $x_1\in \neg U_2$, then $x_1\in U_1\cap\neg U_2$, so we take $y=x_1$ and $i=1$. If $x_1\not\in \neg U_2$, then there is an $x_2\geqslant x_1$ such that $x_2\in U_2$. By transitivity, $x_2\geqslant x$. If $x_2\in\neg U_3$, then $x_2\in U_2\cap \neg U_3$, so we take $y=x_2$ and $i=2$. If $x_2\not\in U_3$, then there is an $x_3\geqslant x_2$ such that $x_3\in U_3$, etc. If we do not find our $y$ and $i$ in this way by $n-1$, then we reason as follows: given $x_{n-1}\not\in \neg U_n$, there is an $x_n\geqslant x_{n-1}$ such that $x_n\in U_n$. Then since $U_{n+1}=\varnothing$, we have $x_n\in U_n\cap\neg U_{n+1}$, so we set $y=x_n$ and $i=n$.\end{proof}

To obtain Corollary \ref{cor: partition2} below from Proposition \ref{prop: partition}, we use the following topological fact.

\begin{fact}\label{TopFact} For any space $X$, $V\subseteq X$, and open $U\subseteq X$, if $U\cap V=\varnothing$, then $U\cap \mathsf{int}(\mathsf{cl}(V))=\varnothing$.
\end{fact}

\begin{cor}\label{cor: partition2}
Let $X$ be a UV-space. For each positive integer $m$, there is a family $\{U_1,\dots, U_m\}$ of pairwise disjoint $\mathsf{CO}\mathcal{RO}$ sets such that:
\begin{enumerate}
\item\label{partition2a} for every $x\in X$, there is a unique $K\subseteq \{1,\dots, m\}$ such that $x\in {\bigvee}_{k\in K}U_k$ and $x\not\in {\bigvee}_{j\in J} U_j$ for each $J\subsetneq K$;
\item\label{partition2b} for every $K\subseteq  \{1,\dots, m\}$ such that $K\neq\varnothing$, there is an $x\in X$ such that $x\in \underset{k\in K}{\bigvee} U_k$ and  $x\not\in \underset{j\in J}{\bigvee} U_j$ for each $J\subsetneq K$;
\item\label{partition2c} for every $K\subseteq  \{1,\dots, m\}$, if $x\not\in \neg \bigvee_{k\in K} U_k \vee \bigvee_{j\in J} U_j $ for each $J\subsetneq K$, then there is a $y\geqslant x$ such that $y\in {\bigvee}_{k\in K} U_k$ and $y\not\in {\bigvee}_{j\in J} U_j$ for each $J\subsetneq K$.
\end{enumerate}
\end{cor}

\begin{proof} By Proposition \ref{prop: partition}, there is a family $\{U_1,\dots, U_m\}$ of pairwise disjoint $\mathsf{CO}\mathcal{RO}(X)$ sets such that $X= U_1\vee\dots\vee U_m$. It follows that for each $x\in X$, there is a $K\subseteq \{1,\dots, m\}$ such that $x\in \bigvee_{k\in K} U_k$ and such that $x\not\in \bigvee_{j\in J} U_j$ for each $J\subsetneq K$. It remains to show that this $K$ is unique. Suppose not, so there is a $K'\subseteq \{1,\dots, m\}$ such that $K'\neq K$, $x\in \bigvee_{k\in K'} U_k$, and $x\not\in \bigvee_{j\in J'} U_j$ for each $J'\subsetneq K'$. Since $K'\not\subseteq K$, pick $k'\in K'\setminus K$. Since $x\in \bigvee_{k\in K'} U_k$ but $x\not\in \bigvee_{j\in K'\setminus \{k'\}} U_j$, it follows that there is some $y\geqslant x$ such that $y\in U_{k'}$. Since $U_{k'}$ is disjoint from $\bigcup_{k\in K} U_k$, we have $y\not\in \bigvee_{k\in K} U_k$ by Fact \ref{TopFact} and the equation for join in terms of $\mathsf{int}_\leqslant \mathsf{cl}_\leqslant$ and union (Proposition \ref{COROBA}). But then since $y\geqslant x$, we have $x\not\in \bigvee_{k\in K} U_k$, contradicting our assumption. Thus, $K$ is unique.

For part 2, let $K\subseteq  \{1,\dots, m\}$ and $K\neq\varnothing$. Since $\{V\in \mathsf{CO}\mathcal{RO}(X)\mid \bigvee_{k\in K}U_k\subseteq V\}$ is a proper filter in $\mathsf{CO}\mathcal{RO}(X)$, it follows by the definition of a UV-space (Definition \ref{VOspace}.\ref{PossCompact}) that there is some $x\in X$ such that $\mathsf{CO}\mathcal{RO}(x)=\{V\in \mathsf{CO}\mathcal{RO}(X)\mid \bigvee_{k\in K}U_k\subseteq V\}$. Now suppose $J\subsetneq K$ and consider some $i\in K\setminus J$. Hence by Fact \ref{TopFact}, $U_i$ is disjoint from $\bigvee_{j\in J}U_j$. Thus, $\bigvee_{k\in K}U_k\not\subseteq \bigvee_{j\in J}U_j$, which with $\mathsf{CO}\mathcal{RO}(x)=\{V\in \mathsf{CO}\mathcal{RO}(X)\mid \bigvee_{k\in K}U_k\subseteq V\}$ implies that $\bigvee_{j\in J}U_j\not\in \mathsf{CO}\mathcal{RO}(x)$, i.e., $x\not\in \bigvee_{j\in J}U_j$.

For part 3, let $K\subseteq  \{1,\dots, m\}$, and suppose  $x\not\in \neg \bigvee_{k\in K} U_k \vee \bigvee_{j\in J} U_j $ for each $J\subsetneq K$. It follows that the filter $F$ in $\mathsf{CO}\mathcal{RO}(X)$ generated by $\mathsf{CO}\mathcal{RO}(x)\cup \{ \bigvee_{k\in K}U_k\}$ is a proper filter such that $\bigvee_{j\in J} U_j\not\in F$ for each $J\subsetneq K$. Then by the definition of a UV-space (Definition \ref{VOspace}.\ref{PossCompact}), there is some $y\in X$ such that $\mathsf{CO}\mathcal{RO}(y)=F$. Hence $\mathsf{CO}\mathcal{RO}(x)\subseteq\mathsf{CO}\mathcal{RO}(y)$, which implies $x\leqslant y$ by the definition of a UV-space (Definition \ref{VOspace}), and $y\in \bigvee_{k\in K}U_k$. Finally, for $J\subsetneq K$, from $\bigvee_{j\in J} U_j\not\in F$ and $F=\mathsf{CO}\mathcal{RO}(y)$, we have $y\not\in \bigvee_{j\in J}U_j$.\end{proof}

\begin{theorem}\label{SubalgebrasProp} Every infinite BA $\mathbb{B}$ contains subalgebras isomorphic to $\mathbb{B}_n$ for each positive integer $n$.
\end{theorem}

\begin{proof} 
Let $X$ be the infinite UV-space dual to $\mathbb{B}$ and $X_n$ the finite UV-space dual to $\mathbb{B}_n$. By duality, it suffices to show there is a surjective UV-map  $f$ from $X$ onto $X_n$.  Let $x_1,\dots, x_m$ be the maximal elements of $X_n$.  By Corollary \ref{UVspectral}.\ref{UVspectral4.5}, $x_1,\dots, x_m$ are the co-atoms of a Boolean algebra obtained by adding a top node to $X_n$. Therefore, we have that (a) for every $y\in X_n$, there is a unique $K\subseteq \{1,\dots,m\}$ such that $y=\bigwedge_{k\in K} x_k$. 

Take a family $\{U_1,\dots, U_m\}$ of $\mathsf{CO}\mathcal{RO}(X)$ sets as in Corollary \ref{cor: partition2}.  By Corollary \ref{cor: partition2}.\ref{partition2a}, for each $x\in X$, there is a unique $K\subseteq\{1,\dots,m\}$ such that $x\in \bigvee_{k\in K} U_k$ and  $x\not\in \bigvee_{j\in J} U_j$ for each $J\subsetneq K$. Let $f(x)=\bigwedge_{k\in K} x_k$. Then by (a) and Corollary \ref{cor: partition2}.\ref{partition2b}, $f$ is surjective.

Now we show that $f$  is a UV-map. First note that the compact opens of $X_n$ are exactly the upsets of $X_n$ with respect to $\leqslant$. Now let $y\in X_n$ be such that  $y=\bigwedge_{i\in I} x_i$ for some $I\subseteq \{1,\dots,m\}$. Then it follows from the definition of 
$f$ that $f^{-1}[{\Uparrow}y] = \bigvee_{i\in I} U_i \in \mathsf{CO}\mathcal{RO}(X)$. Now let $U\subseteq X_n$. Then $U = \bigcup_{y\in U} {\Uparrow}y$ and $f^{-1}[U] = f^{-1} [ \bigcup_{y\in U} {\Uparrow}y] = \bigcup_{y\in U} f^{-1}[{\Uparrow}y]$. 
Since the collection of compact open sets is closed under finite unions, we obtain that $f^{-1}[U] $ is compact open in $X$. Therefore, $f$ is a spectral map. 

Finally, suppose $f(x)\leqslant y'$.  Then there are $I, K\subseteq \{1,\dots,m\}$ such that $I\subseteq K$, $y' =   \bigwedge_{i\in I} x_i$, $f(x) = \bigwedge_{k\in K} x_k$, and $K$ is the unique subset of $\{1,\dots,m\}$ such that $x\in \bigvee_{k\in K}U_k$ and (b) $x\not\in \bigvee_{j\in J}U_j$ for each $J\subsetneq K$. We claim  there is a $y\geqslant x$ such that (c) $y\in \bigvee_{i\in I}U_i$ and $y\not\in \bigvee_{\ell\in L}U_\ell$ for each $L\subsetneq I$. By Corollary \ref{cor: partition2}.\ref{partition2c}, it suffices to show that $x\not\in \neg \bigvee_{i\in I}U_i\vee \bigvee_{\ell\in L}U_\ell$ for each $L\subsetneq I$. For contradiction, suppose $x\in \neg \bigvee_{i\in I}U_i\vee \bigvee_{\ell\in L}U_\ell$ for some $L\subsetneq I$. Then since $I\subseteq K$ and $x\in \bigvee_{k\in K}U_k$, it follows that $x\in \bigvee_{k\in K\setminus (I\setminus L)}U_k$, which contradicts (b). By (c) and the definition of $f$, we have $f(y)=\bigwedge_{i\in I} x_i=y'$. Thus, $f$ is a UV-map.\end{proof} 

\section{Perspectives on UV-spaces assuming choice}\label{ChoiceSection} In this penultimate section, we briefly discuss some results about UV-spaces that can be proved under the assumption of the Boolean Prime Ideal Theorem (BPI).

\subsection{UV-spaces as upper Vietoris spaces of Stone spaces}\label{UVStoneSection} Recall from Definition \ref{UVStoneDef} that for a Stone space $X$, $\mathscr{UV}(X)$ is the hyperspace of nonempty closed subsets of $X$ endowed with the upper Vietoris topology. We already observed (Corollary \ref{StoneCor}) that $\mathscr{UV}(X)$ is a UV-space. Assuming the BPI, every UV-space arises homeomorphically in this way.

\begin{prop} Assuming the BPI, every UV-space is homeomorphic to $\mathscr{UV}(X)$ for some Stone space $X$.
\end{prop}

\begin{proof} Let $Y$ be a UV-space. Let $X$ be the Stone dual of $\mathsf{CO}\mathcal{RO}(Y)$, so $\mathsf{Clop}(X)$ is isomorphic to $\mathsf{CO}\mathcal{RO}(Y)$. By Proposition \ref{UVStone}, $\mathscr{UV}(X)$ is homeomorphic to $UV(\mathsf{Clop}(X))$. Combining the previous two facts, we have that $\mathscr{UV}(X)$ is homeomorphic to $UV(\mathsf{CO}\mathcal{RO}(Y))$, which is homeomorphic to $Y$ by Theorem \ref{SecondThm}.\ref{SecondThmB}. Thus, $Y$ is homeomorphic to $\mathscr{UV}(X)$.\end{proof}

\subsection{Equivalent Priestley spaces assuming choice}\label{Priestley} In this subsection, we relate UV-spaces to Priestley spaces \cite{Priestley1970}. For convenience, we now keep track of the topology $\tau$ of a space explicitly.

For a spectral space $(X,\tau)$, its corresponding Priestley space $(X, \leq, \tau^+)$ is defined as follows: $\tau^+$ is the patch topology of $\tau$, i.e., the topology generated by $\tau \cup\{X\setminus U \mid U\in \tau\}$ as a subbasis, and $\leq$ is the specialization order of $\tau$. Note that the BPI, in its equivalent form as the Alexander Subbasis Theorem, is used already in showing that the patch topology of a spectral topology is 
compact. Conversely,  if $(X, \leq, \tau)$ is a Priestley space, then $X$ together with the topology given  
by open upsets of $(X, \leq, \tau)$ is a spectral space.  It is well known that $U\subseteq X$ is compact open in a spectral space iff $U$ is a clopen upset in the associated Priestley space, with the right-to-left direction using the~BPI.

Since UV-spaces are spectral, given a UV-space $(X, \tau)$ we can consider the corresponding Priestley space  $(X, \leq, \tau^+)$. It is easy to see that $U\subseteq X$ is 
$\mathsf{CO}\mathcal{RO}$ in $(X, \tau)$ iff $U$ is a clopen $\mathcal{RO}$ subset of $(X, \leq, \tau^+)$, where $\mathcal{RO}$ is now taken with respect to $\leq$. 
Let $\mathsf{Clop}\mathcal{RO}(X)$ be the set of clopen $\mathcal{RO}$ subsets of  $(X, \leq, \tau)$.
Then the definition of a UV-space easily translates into 
the following definition in terms of Priestley spaces.   

\begin{definition}
A Priestley space $(X, \leq, \tau)$ is a \emph{UV-Priestley space}  iff:
\begin{enumerate}
\item $\mathsf{Clop}\mathcal{RO}(X)$ is closed under  $\mathsf{int}_\leq (X\setminus \cdot)$;

\item if $x\nleq y$, then there is a $U\in  \mathsf{Clop}\mathcal{RO}(X)$ such that $x\in U$ and $y\notin U$.\footnote{Note that if we had only required that $U$ be a clopen upset, then part 2 would be exactly the Priestley separation axiom.} 

\item every proper filter in $\mathsf{Clop}\mathcal{RO}(X)$ is $\mathsf{Clop}\mathcal{RO}(x)$ for some $x \in  X$, where $\mathsf{Clop}\mathcal{RO}(x) = \{U\in \mathsf{Clop}\mathcal{RO}(X) \mid x\in U\}$.
\end{enumerate}
\end{definition}

It is easy to verify that if $(X, \tau)$ is a UV-space, then $(X, \leq,\tau^+)$ is a UV-Priestley space, and if
$(X, \leq,\tau)$ is a UV-Priestley space, then $X$ together with the topology given by open upsets of $(X, \leq,\tau)$ is a UV-space. 
Moreover, given a UV-Priestley space $(X, \leq,\tau)$, it is easy to see that $\mathsf{Clop}\mathcal{RO}(X)$ is a BA with meet as intersection and $\neg U = \mathsf{int}_\leqslant (X\setminus U)$. Conversely, given a BA $\mathbb{A}$, we obtain a dual UV-Priestley space $X$ based on the set of proper filters in $\mathbb{A}$ by defining $\leq$ as $\subseteq$ and generating a topology by declaring $\{\widehat{a}, \mathrm{PropFilt}(\mathbb{A})\setminus \widehat{a}\mid a\in \mathbb{A}\}$ as a subbasis. It is easy to see that this is the same as taking the UV-space dual 
to $\mathbb{A}$ and considering its corresponding UV-Priestley space. Then $\mathbb{A}$ is isomorphic to the BA $\mathsf{Clop}\mathcal{RO}(X)$, and each UV-Priestley space $Y$ is order-homeomorphic to the dual of $\mathsf{Clop}\mathcal{RO}(Y)$.

Next we discuss morphisms, which are the obvious adaptation of the UV-maps of Definition \ref{UVmapDef} to the Priestley setting.

\begin{definition}
 A map $f:X\to X'$ between UV-Priestley spaces is called a \emph{UV-Priestley morphism} iff it is a Priestley morphism (i.e., continuous and order-preserving) satisfying the p-morphism condition:
 \begin{center} if $f(x) \leq' y'$, then  $\exists y$: $x \leq  y$ and $f(y)=y'$.
 \end{center}
\end{definition}

Assuming the BPI, it is easy to show that the category of UV-spaces and UV-maps is isomorphic to the category of UV-Priestley spaces and UV-Priestley morphisms, which is therefore dually equivalent to the category of BAs and BA homomorphisms by Theorem \ref{DualityThm}.

One can also develop a duality dictionary for this duality similar to the one discussed in Section 6. But we will not do so here, as our primary goal is to study the setting of choice-free dualities for BAs.

Just as one can move freely between Priestley spaces and the \textit{pairwise Stone spaces} of \cite{BBGK10}, one can also move freely between UV-Priestley spaces and analogous pairwise UV-spaces. We omit the details, as they are straightforward to reconstruct based on the information above and in \cite{BBGK10}.
 
\subsection{Goldblatt's representation of ortholattices}\label{GoldblattSection} Our choice-free duality for BAs is related to Goldblatt's \cite{Goldblatt1975} representation of ortholattices. An ortholattice is a bounded lattice equipped with an additional unary operation $'$ such that $a\wedge a'=\zero$, $a\vee a'=1$, $a''=a$, and $a\leq b$ only if $b'\leq a'$. Goldblatt showed that ortholattices can be represented using a Stone space $X$ equipped with a symmetric and irreflexive relation $\bot$. A subset $U\subseteq X$ is \textit{$\bot$-regular} iff $U=U^{\bot\bot}$ where $V^\bot=\{x\in X\mid x\bot y \mbox{ for all }y\in V\}$. The collection of all $\bot$-regular subsets ordered by inclusion forms a complete ortholattice with $'$ as $^\bot$. Conversely, every complete ortholattice is isomorphic to the collection of $\bot$-regular subsets with $^\bot$ coming from a set with a symmetric and irreflexive relation $\bot$. To represent an arbitrary ortholattice $L$, Goldblatt defined a space $X$ with a binary relation $\bot$ as follows:
\begin{enumerate}
\item the underlying set of $X$ is the set of all proper filters of $L$;
\item for $F,G\in X$, let $F \bot G$ iff there is some $a\in F$ such that $a'\in G$;
\item the topology of $X$ is generated by the collection of sets $\widehat{a}$ and $X\setminus \widehat{a}$ as a subbasis, i.e., the patch topology associated with $\tau=\{\widehat{a}\mid a\in L\}$.
\end{enumerate}
Assuming the BPI, Goldblatt proved that $X$ is a Stone space and $L$ is isomorphic to the collection of \textit{clopen $\bot$-regular} sets ordered by inclusion with the operation $^\bot$. Since every BA is an ortholattice with $'$ as Boolean complement, this representation applies to BAs. Like our representation of BAs, it uses the proper filters of $L$. Indeed, Goldblatt's representation applied to BAs is essentially the UV-Priestley representation discussed in Section \ref{Priestley} but using the incompatibility relation $\bot$ between proper filters instead of the inclusion order on proper filters, which is the specialization order of $\tau$. It is easy to see that for a BA, the $\bot$-regular sets are exactly the regular open sets with respect to the inclusion order. 

There are two important differences between Goldblatt's representation applied to BAs and ours. First, because we work with the spectral topology $\tau$ instead of the patch topology, we do not need the extra datum of the relation $\bot$; the regular sets can be defined simply in terms of the specialization order of the space. Thus, we can work with spaces instead of spaces plus a binary relation. Second, because we work with the spectral topology $\tau$ instead of the patch topology, we do not require the nonconstructive BPI.

\section{Conclusion}\label{Conclusion} We have developed a full choice-free duality for BAs in terms of UV-spaces. We showed how to translate, via this duality, the main algebraic concepts and constructions into topological terms. We also gave several sample applications of this duality in the form of choice-free proofs, using spatial intuition essentially, of some basic facts about BAs.

The distinguishing features of the duality for BAs in this paper are that (a) the duals of BAs are topological spaces and (b) the duality is choice-free. Standard Stone duality satisfies (a) but not (b). The pointfree duality using Stone locales satisfies (b) but not (a). To draw a contrast with the localic approach, we characterized our approach to choice-free Stone duality as the \textit{hyperspace approach}. The choice-freeness is achieved by not working with Stone spaces, but rather with UV-spaces, examples of which are given by the upper Vietoris hyperspace of a Stone space. Assuming choice, all UV-spaces arise homeomorphically in this way; but we do not need this assumption to carry out our duality for BAs.

Though we have concentrated on BAs, we believe that choice-free
duality does not end here. In future work, we aim to generalize the strategy of
this paper to obtain choice-free spatial dualities for other classes of algebras (connecting with work in \cite{Massas2016}), giving rise to choice-free completeness proofs for non-classical logics. We hope that this can be the beginning of a new area of choice-free
duality in non-classical logic and beyond.\\

\subsection*{Acknowledgement}
For helpful feedback, we wish to thank Johan van Benthem, Benno van den Berg, Guram Bezhanishvili, Yifeng Ding, David Gabelaia, Tom\'{a}\v{s} Jakl, Mamuka Jibladze, Frederik Lauridsen, Vincenzo Marra, Shezad Mohamed, Floris Sluijter, Joran van Weel, and the anonymous referee for \textit{The~Journal~of~Symbolic~Logic}. We are also grateful for comments from audiences at ALCOP 2017 in Glasgow, the 2017 Algebra and Coalgebra Seminar at the ILLC, University of Amsterdam, ToLo 2018 in Tbilisi, and BLAST 2018 in Denver.

 \bibliographystyle{asl}
  \bibliography{choicefree}

\end{document}